\RequirePackage{fix-cm}
\documentclass{svjour3}                     
\smartqed  
\usepackage{graphicx}
\usepackage{latexsym,url,amscd,amsmath,amssymb}
\spnewtheorem{assumption}{Assumption}{\bf}{\it}
\usepackage{tikz}
\usetikzlibrary{matrix}
\usepackage{algorithm}
\usepackage{algorithmic}
\DeclareMathOperator*{\argmin}{argmin}
\usepackage{multirow}
\usepackage{framed}
\usepackage{float} 
\usepackage{subfigure}
\usepackage{courier}
\begin{document}

\title{Lifted Stationary Points of Sparse Optimization with Complementarity Constraints 
}


\author{Shisen Liu         \and
        Xiaojun Chen 
}


\institute{
Shisen Liu \at
              Department of Applied Mathematics, The Hong Kong Polytechnic University, Hong Kong, P.R. China,               \email{shisen.liu@connect.polyu.hk}           
           \and
           Xiaojun Chen \at
              Department of Applied Mathematics, The Hong Kong Polytechnic University, Hong Kong, P.R. China,
              \email{maxjchen@polyu.edu.hk}\\
}

\date{9 December 2022}

\maketitle

\begin{abstract}
We aim to compute lifted stationary points of a sparse optimization problem (\ref{P0}) with complementarity constraints. 
We define a continuous relaxation problem (\ref{Rv}) that has the same global minimizers and optimal value with problem (\ref{P0}). Problem (\ref{Rv}) is a mathematical program with complementarity constraints (MPCC) and a difference-of-convex (DC) objective function. We define MPCC lifted-stationarity of (\ref{Rv}) and show that it is weaker than directional stationarity, but stronger than Clarke stationarity for local optimality. Moreover, we propose an approximation method to solve (\ref{Rv}) and an augmented Lagrangian method to solve its subproblem (\ref{Rvsig}), which relaxes the equality constraint in (\ref{Rv}) with a tolerance $\sigma$.  We prove the convergence of our algorithm to an MPCC lifted-stationary point of problem (\ref{Rv}) and use a sparse optimization problem with vertical linear complementarity constraints to demonstrate the efficiency of our algorithm on finding sparse solutions in practice.
\keywords{Sparse solution \and Complementarity constraints \and Capped-$\ell_1$ folded concave function \and Lifted stationary point \and Vertical linear complementarity constraints}
\end{abstract}

\section{Introduction}
\label{intro}
Let $G$ and $H$ be continuously differentiable functions from $\mathbb{R}^n$ to $\mathbb{R}^m$. The nonlinear complementarity system (NCS) is to find a vector $x\in\mathbb{R}^n$ such that
\begin{equation}\label{GNCP}
  G(x)\geq0,\ H(x)\geq0,\ G(x)^TH(x)=0.
\end{equation}
 Sparse solutions of systems of equalities and inequalities have been studied in economics and finance \cite{s1,s2,s5}, engineering \cite{s2,s13}, data science \cite{s29,s30,YCX} and signal processing \cite{s32,s33}. The following problem is to find sparse solutions of NCS (\ref{GNCP}):
\begin{equation}\tag{$P$\textsubscript{0}}\label{P0}
\begin{split}
\min\ & \|x\|_0\\
{\rm s.t.}\ & G(x)\geq0,\ H(x)\geq0,\ G(x)^TH(x)=0,
\end{split}
\end{equation}
where $\|x\|_0$ denotes the number of nonzero components of $x\in\mathbb{R}^n$.\\

Problem (\ref{P0}) is a mathematical program with complementarity constraints (MPCC) and cardinality objective, which includes sparse optimization for finding sparse solutions of linear equations (LE) and linear complementarity problems (LCP) as special cases. It is known that finding a sparse solution of a system of linear equations is NP-hard \cite{s8}. Cand${\rm\grave{e}}$s and Tao \cite{s8} introduced the restricted isometry property (RIP) and restricted orthogonality (RO) and proved that under RIP and RO on the coefficient matrix, a sparse solution of LE can be obtained by replacing $\|\cdot\|_0$ by $\|\cdot\|_1$. Chen and Xiang \cite{s3} utilized $L_p$ ($0<p<1$) norm to replace $\|\cdot\|_0$ to get a sparse solution of LCP,  and proved that there exists a positive lower bound $\bar{p}$ such that any least-$p$-norm solution is a sparse solution of LCP when $0<p<\bar{p}$.
 \par In \cite{s31}, Bian and Chen considered a box-constrained optimization problem with cardinality penalty, and used capped-$\ell_1$ folded concave function $\phi(t):=\min\{1,\frac{|t|}{\nu}\}, t\in\mathbb{R}$ with $\nu>0$ to replace $\|\cdot\|_0$. In \cite{s34}, Pan and Chen considered convex-constrained group sparse optimization for image recovery using different capped folded concave functions.
   In this paper, we replace $\|x\|_0$ by the capped-$\ell_1$ folded concave function and consider a sparse optimization problem with complementarity constraints as follows:
\begin{equation}\tag{$R$\textsubscript{$\nu$}}\label{Rv}
\begin{split}
\min\ & \Phi(x):=\sum_{i=1}^{n}\phi(x_i)\\
{\rm s.t.}\ & G(x)\geq0,\ H(x)\geq0,\ G(x)^TH(x)=0.
\end{split}
\end{equation}
\par Note that the function $\phi(\cdot)$ can be reformulated as a difference of convex (DC) function, i.e.
\begin{equation*}
  \phi(t)=\frac{|t|}{\nu}-\max\{\theta_1(t),\theta_2(t),\theta_3(t)\},
\end{equation*}
where $\theta_1(t)=0$, $\theta_2(t)=t/\nu-1$ and $\theta_3(t)=-t/\nu-1$. This implies that problem (\ref{Rv}) is a DC program with nonlinear complementarity constraints. The history of DC programs can date back to the work of Pham Dinh and Tao in 1985 \cite{RecentDC}. Pang, Razaviyayn and Alvarado \cite{s68} proposed a novel iterative algorithm for obtaining a d(irectional)-stationary point of a convex-constrained DC program. To our best knowledge, there exists little work on solving complementarity-constrained DC programs. This motivates us to find d-stationary points and lifted stationary points of MPCC, which define stronger stationarity than C-stationarity. (Formal definitions of a d-stationary point
and a lifted stationary point are introduced in subsection \ref{subneccon} with references \cite{directionalCui,directionalD,s68}.)
\par Problem (\ref{Rv}) is a locally Lipschitz MPCC and the work related to MPCC or mathematical programs with equilibrium constraints (MPEC) can be found in \cite{s42,s40,s41,s43,s61}. Recently, Guo and Chen \cite{s44} considered a class of MPCC with a non-Lipschitz continuous objective for obtaining sparse solutions of NCS.
The objective function of problem (\ref{Rv}) can be represented as a difference of two piecewise linear convex functions, which inspires us to develop an efficient method to find sparse solutions of NCS. In this paper, we use an approximation method proposed in \cite{s46} to relax the complementarity constraints and consider the following approximation problem:
\begin{equation}\tag{$R$\textsubscript{$\nu,\sigma$}}\label{Rvsig}
\begin{split}
\min\ & \Phi(x)\\
{\rm s.t.}\ &  G(x)\geq0,H(x)\geq0,G(x)^TH(x)\leq\sigma,
\end{split}
\end{equation}
where $\sigma>0$ is a tolerance parameter. Notice that convergence analysis in \cite{s46} is for a smooth objective function, but the objective function of problem($R_{\nu,\sigma})$ is nonsmooth. Hence we need to develop new convergence results for this scheme.
\par Our contributions can be summarized as follows:
\begin{itemize}
  \item Based on the equivalence between problem (\ref{P0}) and problem (\ref{Rv}) with certain $\nu$, we prove that global minimizers of problem (\ref{Rvsig}) converge to some global minimizers of problem (\ref{P0}) as $\sigma\downarrow0$. Besides, we provide an upper bound for the distance between a feasible point of problem (\ref{Rvsig}) and the feasible set of problem (\ref{Rv}), which can be used to estimate the distance between the solution sets of these two problems under certain conditions. Furthermore, we establish the relationship between MPCC lifted-stationary point of problem (\ref{Rv}) and different kinds of stationary points of problem (\ref{Rv}). See Fig. \ref{figure1}.

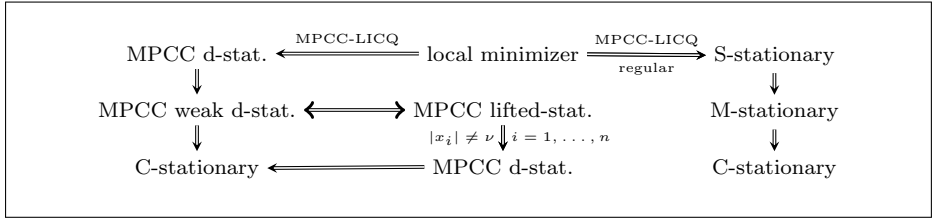
\begin{figure}[H]
\centering
\begin{framed}
\begin{tikzpicture}
  \matrix (m) [matrix of math nodes,row sep=1em,column sep=4.5em,minimum width=1em]
  {
       \text{MPCC d-stat.} &\text{local minimizer} & \text{S-stationary}\\
       \text{MPCC weak d-stat.}&\text{MPCC lifted-stat.}&\text{M-stationary}\\
       \text{C-stationary}&\text{MPCC d-stat.}&\text{C-stationary}\\};
  \path[-stealth]
    (m-1-2) edge [double] node [above] {{\tiny MPCC-LICQ}} (m-1-1)
            edge [double] node [above] {{\tiny MPCC-LICQ}} (m-1-3)
            edge [double] node [below] {{\tiny regular}} (m-1-3)
    (m-1-1) edge [double] node {} (m-2-1)
    (m-1-3) edge [double] node {} (m-2-3)
    (m-2-1) edge [double,<->] node {} (m-2-2)
            edge [double] node {} (m-3-1)
    (m-2-2) edge [double] node [left] {{\tiny $|x_i|\neq\nu$}} (m-3-2)
            edge [double] node [right] {{\tiny $i=1,\ldots,n$}} (m-3-2)
    (m-3-2) edge [double] node {} (m-3-1)
    (m-2-3) edge [double] node {} (m-3-3);
    \end{tikzpicture}
\end{framed}
\caption{The relationship between different kinds of stationary points of problem (\ref{Rv})} \label{figure1}
\end{figure}
  \item We propose an algorithm for solving problem (\ref{Rv}) and show that under MPCC linear independence constraint qualification (LICQ), any accumulation point $x^*$ of approximate stationary points of problem (\ref{Rvsig}) is an MPCC lifted-stationary point of problem (\ref{Rv}) as $\sigma\downarrow0$. Moreover, we propose an augmented Lagrangian (AL) method \cite{s64} to solve subproblem (\ref{Rvsig}). To the best of our knowledge, the AL method is the first algorithm to find a lifted stationary point of MPCC with a nonsmooth objective function.
  \item We consider a special case of problem (\ref{P0}) where the functions $G$ and $H$ are affine functions. In this case,  the subproblem in the AL method  can be solved by a classical DC algorithm \cite{RecentDC} under certain conditions. We conduct numerical experiments to compare our method by using the capped-$\ell_1$ folded concave function and the $L_p$-minimization ($0<p\leq 1$) method for obtaining sparse solutions of a vertical linear complementarity system. The numerical results demonstrate that our algorithm can find more sparse solutions than the $L_p$-minimization method.
\end{itemize}
\par The rest of this paper is organized as follows. In Section 3, we define the MPCC d-stationary point and the MPCC lifted-stationary point of problem (\ref{Rv}) and give necessary optimality conditions of problem (\ref{Rv}). We also provide an upper bound of the distance between a feasible point of problem (\ref{Rvsig}) and the feasible set of problem (\ref{P0}). In Section 4, we propose an approximation method for problem (\ref{Rv}) and an augmented Lagrangian method for subproblem (\ref{Rvsig}) with convergence analysis, respectively. In Section 5, we consider a special case of problem (\ref{P0}) where the feasible set is the solution set of a vertical linear complementarity system(VLCS) and present numerical results. We conclude the paper in Section 6.

\section{Notation}
The solution set and feasible set of problem (\ref{GNCP}) are respectively defined as follows:
\begin{equation*}
      \mathcal{S}=\{x: G(x)\geq0, H(x)\geq0, G(x)^TH(x)=0\},\ \mathcal{F}=\{x: G(x)\geq0, H(x)\geq0\}.
\end{equation*}
For a vector $x\in\mathbb{R}^n$, $x_+:=\max\{x,0\}=(\max(x_1,0), \ldots, \max(x_n,0))^\top$, ${\rm diag}(x)$ is the diagonal matrix whose $i$th diagonal entry is $x_i$. We denote Euclidean norm by $\|x\|$, $\ell_1$ norm by $\|x\|_1$ and $\ell_0$ norm by $\|x\|_0=\sum_{i=1}^n|x_i|^0$ where $|x_i|^0=\left\{\begin{matrix}
                  1, x_i\neq0, \\
                  0, x_i=0.
                \end{matrix}\right.$
For a given $\nu>0$, let
$$ \Gamma_1(x)=\{i:|x_i|<\nu, \, x_i\neq 0\}\ \text{and}\ \Gamma_2(x)=\{i:|x_i|\geq\nu,\, x_i\neq 0\}.
 $$
The support set $\Gamma(x)$ of $x\in\mathbb{R}^n$ is denoted by
 $$\Gamma(x)=\{i: \,  x_i\neq0\}=\Gamma_1(x)\cup\Gamma_2(x).$$
The distance from $x\in\mathbb{R}^n$ to a closed set $\Omega\subseteq\mathbb{R}^n$ is defined by ${\rm dist}(x,\Omega)=\inf\{\|x-y\|:y\in\Omega\}$ and the indicator function is defined by $\textbf{1}_{\Omega}(x)=\left\{\begin{matrix}
                  1, x\in\Omega, \\
                  0, x\notin\Omega.
                \end{matrix}\right.$ Given a point $x\in\mathbb{R}^n$ and $\delta>0$, $\mathcal{B}_{\delta}(x)$ denotes a closed ball centered at $x$ with radius $\delta$. Let $\boldsymbol{e}_i\in\mathbb{R}^n$ denote the $i$th column of the $n$ dimensional identity matrix. Given a matrix $D\in\mathbb{R}^{r\times l}$, $D_i$ denotes the transpose of the $i$th row of $D$ for $i\in\{1,\ldots,r\}$. We let $\nabla F(x)$ stand for the transposed Jacobian of a smooth function $F$ at $x$.
\par For a point $x\in\mathcal{S}$, we define the following index sets:
\begin{equation}
  \begin{split}
     &\mathcal{I}_G(x)=\{i\in\{1,\ldots,m\}:G_i(x)=0\},\\
     &\mathcal{I}_H(x)=\{i\in\{1,\ldots,m\}:H_i(x)=0\},\\
     &\mathcal{I}_{0+}(x):=\mathcal{I}_G(x)\setminus\mathcal{I}_H(x)=\{i: \, G_i(x)=0,\ H_i(x)>0\},\\
     &\mathcal{I}_{+0}(x):=\mathcal{I}_H(x)\setminus\mathcal{I}_G(x)=\{i: \, G_i(x)>0,\ H_i(x)=0\},\\
     &\mathcal{I}_{00}(x):=\mathcal{I}_G(x)\cap\mathcal{I}_H(x)=\{i: \, G_i(x)=0,\ H_i(x)=0\}.\\
  \end{split}
\end{equation}
Given a closed set $\Omega$ and a point $x^*\in\Omega$, the regular normal cone \cite{s4} of $\Omega$ at $x^*$ is a closed and convex cone defined as
\begin{equation*}
  \hat{\mathcal{N}}_{\Omega}(x^*):=\{u: u^T(x-x^*)\leq o(\|x-x^*\|), \forall x\in\Omega\},
\end{equation*}
where $o(\cdot)$ means that $o(\alpha)/\alpha\rightarrow0$ as $\alpha\downarrow0$, and limiting normal cone \cite{s64} of $\Omega$ at $x^*$ is a closed cone defined as
\begin{equation*}
  \mathcal{N}_{\Omega}(x^*):=\{u: \exists\ x^k\in\Omega, x^k\rightarrow x^*, \exists\ u^k\in\hat{\mathcal{N}}_{\Omega}(x^k),\ u^k\rightarrow u\}.
\end{equation*}
For a continuous function $\varphi:\mathbb{R}^n\rightarrow \mathbb{R}$ and a point $x^*\in\mathbb{R}^n$, the regular subdifferential \cite{s4} of $\varphi$ at $x^*$ is defined as
\begin{equation*}
  \hat{\partial}\varphi(x^*):=\{v: \varphi(x)\geq\varphi(x^*)+v^T(x-x^*)+o(\|x-x^*\|), \forall x\in\mathbb{R}^n\},
\end{equation*}
the limiting subdifferential \cite{s4} of $\varphi$ at $x^*$ is defined as
\begin{equation*}
  \partial\varphi(x^*):=\{v: \exists\ x^k\rightarrow x^*, v^k\in\hat{\partial}_{\varphi}(x^k),\ v^k\rightarrow v\}.
\end{equation*}
 If $\Omega$ is a convex set, then $\mathcal{N}_{\Omega}(x)=\hat{\mathcal{N}}_{\Omega}(x)$, see \cite{s4}.

\begin{lemma}\label{lemma1.1}{(Lemma 1.1 in \cite{s44})}
  Let $x\in\mathcal{S}$. Then for any $(u,v)\in\mathbb{R}^m\times\mathbb{R}^m$ with $u_i=0$ $i\in\mathcal{I}_{+0}(x)$, $v_i=0$ $i\in\mathcal{I}_{0+}(x)$, $u_i\geq0, v_i\geq0$ $i\in\mathcal{I}_{00}(x)$, we have that
  \begin{equation*}
    -\nabla G(x)u-\nabla H(x)v\in\mathcal{N}_{\mathcal{S}}(x).
  \end{equation*}
\end{lemma}

If $0\in\mathcal{S}$, then 0 is the unique global minimizer of problem (\ref{P0}). Without loss of generality, we give the following assumption.
\begin{assumption}\label{assump1}
      The solution set $\mathcal{S}$ of problem (\ref{GNCP}) is nonempty and $0\not\in \mathcal{S}.$
\end{assumption}
Throughout Sections 3-4, we assume Assumption \ref{assump1} holds. In Section 5, we provide sufficient conditions for Assumption \ref{assump1} to hold.

\section{Problem (\ref{P0}) and its relaxation problems (\ref{Rv}) and (\ref{Rvsig})}
Since the set $\mathcal{S}$ is closed and $\Phi$ is piecewise linear with $\Phi(x)\ge 0$ for $x\in \mathcal{S}$, problems (\ref{P0}) and (\ref{Rv}) have global minimizers under Assumption \ref{assump1}. First of all, we give the equivalence of global optimality between (\ref{P0}) and (\ref{Rv}).

\begin{lemma}\label{EP0Rv}
There is a $\bar{\nu}>0$ such that problems (\ref{P0}) and (\ref{Rv}) have same global minimizers and same optimal value for any $0<\nu<\bar{\nu}$.
\end{lemma}
\par The proof of Lemma \ref{EP0Rv} is similar to the proof of Theorem 2.1 in \cite{s34} where the feasible set is defined by the intersection of a polyhedron and a possibly degenerate ellipsoid. Since constraints in problems (\ref{P0}) and (\ref{Rv}) are complementarity constraints, we give a proof for completeness in Appendix A.
\begin{remark}
To estimate $\bar{\nu}$ numerically, we
can use the penalty method. Let $\Theta:R^n\to R^{2m+1}$ be defined as $$\Theta(x)=(-G(x)^T, -H(x)^T, G(x)^TH(x))^T.$$
Then $x$ is a feasible point of ($P_0$) and $(R_\nu$) if and only if $\Theta(x)\le 0$. Suppose that $\Theta$ is Lipschitz continuous with Lipschitz constant $L_\Theta>0$ and there is $\lambda>0$ such that problem ($P_0$) and the penalty problem
\begin{equation}\tag{$R$\textsubscript{1}}\label{penalty}
\min \|x\|_0 +\lambda \|\Theta(x)_+\|_1
\end{equation}
have same minimizers and same optimal values. By Theorem 2.4 in \cite{s31}, for any $\nu \in (0, \bar{\nu})$  with $\bar{\nu}= 1/(\lambda L_\Theta)$, problem (\ref{penalty}) and the following problem
\begin{equation}\tag{$R$\textsubscript{2}}\label{penalty-nu}
\min \Phi(x) +\lambda \|\Theta(x)_+\|_1
\end{equation}
have same minimizers and same optimal values. Under such assumptions, we can have Lemma \ref{EP0Rv} with $\bar{\nu}= 1/(\lambda L_\Theta)$. Moreover, we can use an updating scheme (\ref{updating}) of $\widetilde{\nu}_k$  in our numerical experiments, instead of using a fixed $\nu$.
\end{remark}

\subsection{Necessary optimality conditions for problem (\ref{Rv})}\label{subneccon}

It is known that MPCCs are difficult optimization problems because many of the standard constraint qualifications (CQs) (like the linear independence and the Mangasarian–Fromovitz CQs, respectively, LICQ and MFCQ for short) are violated at any feasible point. MPCC constraint qualifications have been proposed and studied in \cite{s60,s40,s41,s62}. Here we introduce several CQs for MPCC that will be adopted in this paper.
\begin{definition}
   \cite{s60} (i) We say that MPCC-LICQ holds at $x^*\in \mathcal{S}$ for problem (\ref{Rv}) if the following gradients are linearly independent:
  \begin{equation}\label{LIgradients}
    \{\nabla G_i(x^*):i\in\mathcal{I}_G(x^*)\}\cup\{ \nabla H_i(x^*):i\in\mathcal{I}_H(x^*)\}.
  \end{equation}
   \par (ii) We say that MPCC-No Nonzero Abnormal Multiplier Constraint Qualification (NNAMCQ) holds at $x^*\in \mathcal{S}$ for problem (\ref{Rv}) if there is no nonzero vector $u,v\in\mathbb{R}^m$ such that
  \begin{equation*}
  \begin{split}
      & 0=\nabla G(x^*)u+\nabla H(x^*)v, \\
       & u_i=0\ {\rm for}\ i\in\mathcal{I}_{+0}(x^*),\ v_i=0\ {\rm for}\ i\in\mathcal{I}_{0+}(x^*),\\
       &  {\rm either}\ u_i> 0,\ v_i>0\ {\rm or}\ u_iv_i=0\ {\rm for}\ i\in\mathcal{I}_{00}(x^*).
  \end{split}
  \end{equation*}
\end{definition}
\par Moreover, we introduce several classical stationarity conditions for problem (\ref{Rv}) as follows, including Clarke (C-), Mordukhovich (M-), and strong (S-) stationarity conditions.
\begin{definition}
   (i) We say that $x^*\in \mathcal{S}$ is a C-stationary point of problem (\ref{Rv}) if there exist $u\in\mathbb{R}^m$ and $v\in\mathbb{R}^m$ such that
  \begin{equation}\label{weakS}
    \begin{split}
        & 0\in \partial\Phi(x^*)-\nabla G(x^*)u-\nabla H(x^*)v, \\
        & u_i=0\ {\rm for}\ i\in\mathcal{I}_{+0}(x^*),\ v_i=0\ {\rm for}\ i\in\mathcal{I}_{0+}(x^*),
    \end{split}
  \end{equation}
and
  \begin{equation*}
    u_iv_i\geq 0\ {\rm for}\ i\in\mathcal{I}_{00}(x^*).
  \end{equation*}
  \par (ii) We say that $x^*\in \mathcal{S}$ is an M-stationary point of problem (\ref{Rv}) if there exist $u\in\mathbb{R}^m$ and $v\in\mathbb{R}^m$ satisfying (\ref{weakS}) and
  \begin{equation*}
    u_i> 0,\ v_i>0\ {\rm or}\ u_iv_i=0\ {\rm for}\ i\in\mathcal{I}_{00}(x^*).
  \end{equation*}
  \par (iii) We say that $x^*\in \mathcal{S}$ is an S-stationary point of problem (\ref{Rv}) if there exist $u\in\mathbb{R}^m$ and $v\in\mathbb{R}^m$ satisfying (\ref{weakS}) and
  \begin{equation*}
    u_i\geq 0,\ v_i\geq0\ {\rm for}\ i\in\mathcal{I}_{00}(x^*).
  \end{equation*}
\end{definition}
For $t\in\mathbb{R}$, denote
\begin{equation*}
  \mathcal{D}(t)=\{i\in\{1,2,3\}: \theta_i(t)=\max\{\theta_1(t),\theta_2(t),\theta_3(t)\}\}.
\end{equation*}

\par Due to the special structure of function $\Phi(\cdot)$, we define the MPCC lifted-stationary point and MPCC d(irectional)-stationary of problem (\ref{Rv}).
\begin{definition}
   We say that $x^*\in \mathcal{S}$ is an MPCC lifted-stationary point of problem (\ref{Rv}) if there exist $d=(d_1,\cdots,d_n)^T$ with $d_i\in\mathcal{D}(x^*_i)$, $i=1,\ldots,n$, and $u\in\mathbb{R}^m$, $v\in\mathbb{R}^m$ such that
  \begin{equation}\label{MPCCd}
    \begin{split}
        & 0\in \partial(\frac{|x^*_i|}{\nu}) -\theta'_{d_i}(x^*_i)-u^T\nabla G_i(x^*)-v^T\nabla H_i(x^*), \\
        & u_i=0\ {\rm for}\ i\in\mathcal{I}_{+0}(x^*),\ v_i=0\ {\rm for}\ i\in\mathcal{I}_{0+}(x^*),\\
        & {\rm and}\ u_iv_i\geq 0\ {\rm for}\ i\in\mathcal{I}_{00}(x^*).
    \end{split}
  \end{equation}
  Moreover, if for any $d=(d_1,\cdots,d_n)^T$ with $d_i\in\mathcal{D}(x^*_i)$, $i=1,\ldots,n$, there exist $u\in\mathbb{R}^m$, $v\in\mathbb{R}^m$ such that (\ref{MPCCd}) holds, then we call $x^*$ an MPCC d-stationary point of problem (\ref{Rv}).
\end{definition}
\par According to \cite{s68}, the MPCC lifted-stationarity can be called as the MPCC weak d-stationarity of problem (\ref{Rv}). It holds that $$\mathcal{X}_d\subseteq\mathcal{X}_{lif}\subseteq\mathcal{X}_{c},$$ but their inverse may not hold, where $\mathcal{X}_d$, $\mathcal{X}_{lif}$ and $\mathcal{X}_{c}$ denote the MPCC d-stationary point set, MPCC lifted-stationary point set and C-stationary point set of problem (\ref{Rv}).
\par
\begin{theorem}\label{NecOptCon}
  Let $x^*$ be a local minimizer of problem (\ref{Rv}).
  \par (i) Assume that $\mathcal{S}$ is regular at $x^*$, that is, $\mathcal{N}_{\mathcal{S}}(x^*)=\hat{\mathcal{N}}_{\mathcal{S}}(x^*)$. If MPCC-LICQ holds at $x^*$, then $x^*$ is an S-stationary point of problem (\ref{Rv}).
  \par (ii) If MPCC-NNAMCQ holds at $x^*$, then $x^*$ is an M-stationary point of problem (\ref{Rv}).
  \par (iii) If MPCC-LICQ holds at $x^*$, then $x^*$ is an MPCC d-stationary point of problem (\ref{Rv}).
\end{theorem}
\begin{proof}
  (i) Since the objective function of problem (\ref{Rv}) is Lipschitz continuous, $x^*$ is an S-stationary point of problem (\ref{Rv}) when $\mathcal{S}$ is regular at $x^*$ and MPCC-LICQ holds at $x^*$ by Theorem 2 in \cite{s61}.
  \par (ii) This statement follows from Theorem 2.3 and Corollary 2.1 in \cite{s62}.
  \par (iii) Since $x^*$ is a local minimizer of problem (\ref{Rv}), we have that $x^*\in\mathcal{S}$ and there exists a constant $\delta>0$ such that for any $x\in\mathcal{S}\cap\mathcal{B}_{\delta}(x^*)$ and any $d=(d_1,\cdots,d_n)^T$ with $d_i\in\mathcal{D}(x^*_i)$, $i=1,\ldots,n$,
  \begin{equation*}
  \begin{split}
     & \sum_{i=1}^{n}(\frac{|x^*_i|}{\nu}-\theta_{d_i}(x^*_i))=\sum_{i=1}^{n}(\frac{|x^*_i|}{\nu}-\max\{\theta_1(x^*_i),\theta_2(x^*_i),\theta_3(x^*_i)\})\\
       =& \Phi(x^*)\leq\Phi(x) \\
       =&\sum_{i=1}^{n}(\frac{|x_i|}{\nu}-\max\{\theta_1(x_i),\theta_2(x_i),\theta_3(x_i)\})\leq\sum_{i=1}^{n}(\frac{|x_i|}{\nu}-\theta_{d_i}(x_i)),
  \end{split}
  \end{equation*}
  where the first equality holds due to $d_i\in\mathcal{D}(x^*_i)$ for $i=1,\ldots,n$, the second equality holds by the definition of $\Phi(\cdot)$, the first inequality holds because $x^*$ is a local minimizer and the last inequality holds due to the function $\max$.\\
  This implies that for any fixed $d=(d_1,\cdots,d_n)^T$ with $d_i\in\mathcal{D}(x^*_i)$, $i=1,\ldots,n$, $x^*$ is a local minimizer of the following program:
  \begin{equation}\label{MPCCdprog}
    \begin{split}
    \min\ & \sum_{i=1}^{n}(\frac{|x_i|}{\nu}-\theta_{d_i}(x_i))\\
    {\rm s.t.}\ & x\in\mathcal{S}.
    \end{split}
  \end{equation}
  \par By Theorem 1 in \cite{s67}, it follows from the proof of Lemma 1 in \cite{s41} that for any $d=(d_1,\cdots,d_n)^T$ with $d_i\in\mathcal{D}(x^*_i)$, $i=1,\ldots,n$, there exist $\alpha\in\mathbb{R}^+$, $u\in\mathbb{R}^m$, $v\in\mathbb{R}^m$, which are not all equal to zero, such that
  \begin{equation*}
    \begin{split}
        & 0\in \alpha(\partial(\frac{|x^*_i|}{\nu}) -\theta'_{d_i}(x^*_i))-u^T\nabla G_i(x^*)-v^T\nabla H_i(x^*), \\
        & u_i=0\ {\rm for}\ i\in\mathcal{I}_{+0}(x^*),\ v_i=0\ {\rm for}\ i\in\mathcal{I}_{0+}(x^*),\\
        & {\rm and}\ u_iv_i\geq 0\ {\rm for}\ i\in\mathcal{I}_{00}(x^*).
    \end{split}
  \end{equation*}
  \par If MPCC-LICQ holds at $x^*$, then it follows from Theorem 2, (1) in \cite{s41} that for any $d=(d_1,\cdots,d_n)^T$ with $d_i\in\mathcal{D}(x^*_i)$, $i=1,\ldots,n$, there exist $u\in\mathbb{R}^m$, $v\in\mathbb{R}^m$ such that (\ref{MPCCd}) holds, which implies that $x^*$ is an MPCC d-stationary point of problem (\ref{Rv}).\qed
\end{proof}

\par Let $0<\nu<\bar{\nu}$, where $\bar{\nu}$ is defined in Lemma \ref{EP0Rv}, such that problems (\ref{P0}) and (\ref{Rv}) have the same global minimizers and optimal value. For fixed $\nu$, denote $\bar{\mathcal{X}}:=\{x\in\mathcal{S}: |x_i|>\nu\ {\rm or}\ x_i=0\ {\rm for}\ i=1,\ldots,n\}$. Then, for any $\bar{x}\in\bar{\mathcal{X}}$, one can easily verify that $\bar{x}$ is a local minimizer of both problems (\ref{P0}) and (\ref{Rv}).

\subsection{Link between (\ref{Rv}) and (\ref{Rvsig})}

Denote the feasible set of problem (\ref{Rvsig}) by
    $$ \mathcal{S}_{\sigma}:=\left\{
 x\in\mathbb{R}^{n}:  G(x)\geq0, H(x)\geq0, G(x)^TH(x)\leq\sigma
\right\},
$$
where $\sigma>0$.\\

Due to Assumption \ref{assump1} and $\mathcal{S}\subseteq\mathcal{S}_{\sigma}$, we know that $\mathcal{S}_{\sigma}$ is nonempty. Moreover, any feasible point of problems (\ref{P0}) is a feasible point of problem (\ref{Rv}) and (\ref{Rvsig}).
\par Let $X^*_\nu$ and $X_{\sigma}$ be the sets of global minimizers of problem (\ref{Rv}) and (\ref{Rvsig}), respectively. It follows from Lemma \ref{EP0Rv} that $X^*_\nu$ is also the set of global minimizers of problem (\ref{P0}) for sufficiently small $\nu$. Let
\begin{equation*}
  X^*:=\{x^*\in\mathbb{R}^n: \exists\ \sigma_k\rightarrow0\ {\rm with}\ \sigma_k>0\ {\rm and}\ x_{\sigma_k}\in X_{\sigma_k}\ {\rm s.t.}\ x_{\sigma_k}\rightarrow x^*\}.
\end{equation*}
Following the idea of \cite{s46}, we give the following proposition. Since the objective is nonsmooth, we give a simple proof for completeness.
  \begin{proposition}\label{EP0sigRvsig}
   Let $x_{\sigma}$ be a global minimizer of problem (\ref{Rvsig}) for $\sigma>0$. Let $x^*$ be an arbitrary accumulation point of $\{x_{\sigma}\}$ as $\sigma\downarrow0$. Then $x^*$ is a global minimizer of (\ref{Rv}), i.e., $X^*\subseteq X^*_\nu$.
  \end{proposition}
  \begin{proof}
   The existence of global minimizer of problem (\ref{Rv}) follows from Assumption 1 and the piecewise linear and nonnegative property of the objective function. Let $\bar{x}$ be a global minimizer of problem (\ref{Rv}). Since $x_{\sigma}$ is a global minimizer of problem (\ref{Rvsig}) and $\mathcal{S}\subseteq\mathcal{S}_{\sigma}$, we have
\begin{equation*}
 G(x_{\sigma})\geq0, H(x_{\sigma})\geq0, G(x_{\sigma})^TH(x_{\sigma})\leq\sigma,
\end{equation*}
and $\Phi(x_{\sigma})\leq\Phi(\bar{x})$.

Let $\sigma$ tend to 0. Since $\Phi(\cdot)$ is continuous, we can obtain that
\begin{equation*}
 G(x^*)\geq0, H(x^*)\geq0, G(x^*)^TH(x^*)\leq0,
\end{equation*}
and $\Phi(x^*)\leq\Phi(\bar{x})$.  This implies that $x^*\in X^*_\nu$. \qed
  \end{proof}
\par The following example shows that Proposition \ref{EP0sigRvsig} may not hold for local minimizers and stationary points.
\begin{example}
Let $m=1$ and $n=2$. Consider the following problem :
\begin{equation}\tag{$E$\textsubscript{0}}\label{exP0}
    \begin{split}
        \min\ & \|x\|_0\\
        {\rm s.t.}\ & x\in\mathcal{S}:=\{z\in\mathbb{R}^2:z_1\geq0,z_2-1\geq0,z_1(z_2-1)=0\}.
    \end{split}
\end{equation}
The approximations of  problem (\ref{exP0}) are as follows, respectively:
\begin{equation}\tag{$E$\textsubscript{$\nu$}}\label{exRv}
    \begin{split}
        \min\ & \Phi(x):=\min(1, |x_1|/\nu) + \min(1, |x_2|/\nu)\\
        {\rm s.t.}\ & x\in\mathcal{S},
    \end{split}
\end{equation}
and
\begin{equation}\tag{$E$\textsubscript{$\nu,\sigma$}}\label{exRvsig}
    \begin{split}
        \min\ & \Phi(x):=\min(1, |x_1|/\nu) + \min(1, |x_2|/\nu)\\
        {\rm s.t.}\ & x\in\mathcal{S}_{\sigma}:=
        \{z\in\mathbb{R}^2:z_1\geq0,z_2-1\geq0,z_1(z_2-1)\leq\sigma\}.
    \end{split}
\end{equation}
For any $0<\nu<\bar{\nu}=0.5$,  problems (\ref{exP0}) and (\ref{exRv}) have the same global minimizers $(0,t)^T,t\geq1$. Let $0<\sigma<\nu$. We consider $\bar{x}_{\sigma}=(\nu+\sigma,\frac{\sigma}{2(\nu+\sigma)}+1)^T$ and its neighborhood $\mathcal{B}_{\bar{\delta}}(\bar{x}_{\sigma})$ with $\bar{\delta}=\frac{\sigma}{4}$, then one can verify that $\bar{x}_{\sigma}\in\mathcal{S}_{\sigma}$ and $\Phi(\bar{x}_{\sigma})=2=\Phi(x)$ for any $x\in\mathcal{S}_{\sigma}\cap\mathcal{B}_{\bar{\delta}}(\bar{x}_{\sigma})$, which implies that $\bar{x}_{\sigma}$ is a local minimizer of problem (\ref{exRvsig}). However, as $\sigma\downarrow0$, $\bar{x}_{\sigma}$ converges to $\bar{x}=(\nu,1)^T$. For any $0<\delta<\nu$, consider $\hat{x}=(\nu-\frac{\delta}{2},1)^T$. We can verify that $\hat{x}\in\mathcal{S}\cap\mathcal{B}_{\delta}(\bar{x})$ and $\Phi(\hat{x})=2-\frac{\delta}{2\nu}<2=\Phi(\bar{x})$, which implies that $\bar{x}$ is not a local minimizer of problem (\ref{exRv}).
\par As for stationarity, one can verify that for any $0<\sigma<\nu$, $\bar{x}_{\sigma}$ is a lifted stationary point of problem (\ref{exRvsig}) and its limit point $\bar{x}$ is an MPCC lifted-stationary point of problem (\ref{exRv}). On the other hand, $\bar{x}_{\sigma}$ is a d-stationary point of problem (\ref{exRvsig}) for $\sigma<\nu$, but    $\bar{x}$ is not an MPCC d-stationary point of problem (\ref{exRv}).
\end{example}

\par Since problem (\ref{Rvsig}) only contains inequality constraints, for the convenience of presentation, we denote
\begin{equation*}
  g(x):=-\begin{pmatrix}
          G(x) \\
          H(x)
        \end{pmatrix}{\rm and}\ h_{\sigma}(x):=G(x)^TH(x)-\sigma.
\end{equation*}
\par Compared with problem (\ref{Rv}), problem (\ref{Rvsig}) can satisfy some constraint qualifications, such as Relaxed Constant Positive Linear Dependence (RCPLD) condition (Definition 2.2, \cite{s64}).

\begin{proposition}\label{3.1}
  (Proposition 3.1, \cite{s44}) Let MPCC-LICQ hold at $x^*\in\mathcal{S}$. Then there exist $\sigma_0>0$ and $\delta_0>0$ such that for any $\sigma\in(0,\sigma_0]$ and  $x\in\mathcal{B}_{\delta_0}(x^*)\cap\mathcal{S}_{\sigma}$, the standard LICQ holds at $x$ for the following system
  \begin{equation*}
    \{x: g(x)\leq0, h_{\sigma}(x)\leq0\}.
  \end{equation*}
  In particular, the standard LICQ holds at $x\in\mathcal{B}_{\delta_0}(x^*)\cap\mathcal{S}_{\sigma}$ with $\sigma\in(0,\sigma_0]$.
\end{proposition}

In the following, we will provide an upper bound for the distance between a feasible point of problem (\ref{Rvsig}) and the feasible set of problem (\ref{Rv}).

\begin{theorem}\label{Errorbound}
    \par (i) Let $x^*\in\mathcal{S}$. If MPCC-NNAMCQ holds at $x^*$, then there exist constants $\delta_0>0$, $\beta\geq0$ such that for any $x\in\mathcal{B}_{\delta_0}(x^*)\cap\mathcal{S}_{\sigma}$,
  \begin{equation*}
  {\rm dist}(x,\mathcal{S})\leq\beta\sigma.
  \end{equation*}
  \par (ii) If both $G$ and $H$ are polynomials, then there exist constants $\beta,\tau>0$ and $\gamma\geq0$ such that for any $x\in\mathcal{S}_{\sigma}$,
  \begin{equation*}
    {\rm dist}(x,\mathcal{S})\leq\beta(1+\|x\|)^{\gamma}\sigma^{\tau}.
  \end{equation*}
  \par (iii) If both $G$ and $H$ are affine functions, then there exists a constant $\beta>0$ such that for any $x\in\mathcal{S}_{\sigma}$,
  \begin{equation*}
    {\rm dist}(x,\mathcal{S})\leq\beta(\sqrt{\sigma}+\sigma).
  \end{equation*}
\end{theorem}
\begin{proof}
  (i) According to Proposition 3.4 in \cite{s65}, we know that there exist constants $\beta\geq0$ and $\delta_0>0$ such that
  \begin{equation}\label{Errorformula1}
    {\rm dist}(x,\mathcal{S})\leq\beta[G(x)^TH(x)]_+,\ \forall\ x\in\mathcal{B}_{\delta_0}(x^*)\cap\mathcal{F}.
  \end{equation}
  Thus, it follows from (\ref{Errorformula1}) that for any $x\in\mathcal{B}_{\delta_0}(x^*)\cap\mathcal{S}_{\sigma}$,
  \begin{equation*}
     {\rm dist}(x,\mathcal{S})\leq\beta[h_{\sigma}(x)+\sigma]_+\leq\beta\sigma,
  \end{equation*}
  which implies that statement (i) holds.
  \par (ii) If both $G,H$ are polynomials, then it follows from Theorem 2.2 in \cite{s57} that there exist constants $\beta,\tau>0$ and $\gamma\geq0$ such that
  \begin{equation*}
    {\rm dist}(x,\mathcal{S})\leq\beta(1+\|x\|)^{\gamma}(\|[g(x)]_+\|+[G(x)^TH(x)]_+)^{\tau},\ \forall x\in\mathbb{R}^n,
  \end{equation*}
  which implies that for any feasible point $x\in\mathcal{S}_{\sigma}$ of problem (\ref{Rvsig}),
  \begin{equation*}
  \begin{split}
      {\rm dist}(x,\mathcal{S})&\leq\beta(1+\|x\|)^{\gamma}(\|[g(x)]_+\|+[h_{\sigma}(x)+\sigma]_+)^{\tau}\\
       & \leq\beta(1+\|x\|)^{\gamma}\sigma^{\tau}.
  \end{split}
  \end{equation*}
  \par (iii) If both $G,H$ are affine functions, then it follows from Theorem 6.1 in \cite{s56} that there exists a constant $\beta_0>0$ such that for any $ x\in\mathbb{R}^n$,
  \begin{equation*}
  \begin{split}
     {\rm dist}(x,\mathcal{S}) & \leq\beta_0[\|\min(G(x),H(x))\|^2+\|\min(G(x),H(x)\| \\
       & +\|G(x)^TH(x)\|+\|G(x)^TH(x)\|^{\frac{1}{2}}].
  \end{split}
  \end{equation*}
  This implies that for any  $x\in\mathcal{S}_{\sigma}$,
  \begin{equation*}
  \begin{split}
      {\rm dist}(x,\mathcal{S})&\leq\beta_0[G(x)^TH(x)+\sqrt{G(x)^TH(x)}+|h_{\sigma}(x)+\sigma|+|h_{\sigma}(x)+\sigma|^{\frac{1}{2}}]\\
       & \leq2\beta_0(\sqrt{\sigma}+\sigma).
  \end{split}
  \end{equation*}
   Let $\beta=2\beta_0$, then the proof is completed.\qed
\end{proof}
\par Theorem \ref{Errorbound} can be used to numerically estimate the distance between a global minimizer  $x_{\sigma}$ of ($R_{\nu,\sigma}$) and the optimum set $X^*_\nu$  of ($R_{\nu}$).  By Proposition \ref{EP0sigRvsig}, $x_{\sigma}$ converges to a solution $x^*$ of problem  ($R_{\nu}$) as $\sigma \to 0$. Let $\bar{x}\in {\cal S}$ satisfy  $\|x_\sigma-\bar{x}\| =$dist$(x_\sigma, {\cal S})$. By Theorem \ref{Errorbound}, and $X^*_\nu \subseteq {\cal S}$,  there is a constant $\beta>0$ such that
$$ {\rm dist}(x_\sigma, X^*_\nu) \ge {\rm dist } (x_\sigma, {\cal S})$$
and
$$ {\rm dist}(x_\sigma, X^*_\nu)\le \|x_\sigma-x^*\|\le{\rm dist } (x_\sigma, {\cal S}) + \|\bar{x}-x^*\|\le \beta\sigma + \|\bar{x}-x^*\|.$$
By Proposition \ref{EP0sigRvsig}, $\|\bar{x}-x^*\|\rightarrow0$ as $\sigma\rightarrow0$. Moreover, if $x^*$ is an isolate point in ${\cal S}$, then $\|\bar{x}-x^*\|=0$ for sufficiently small $\sigma$.

\section{Algorithm for problem (\ref{Rv})}
\par In Section 3, it has been pointed out that any accumulation point of global minimizers of problem (\ref{Rvsig}) is a global minimizer of problem (\ref{Rv}) as $\sigma$ approaching to 0. Thus, in this section, we divide our algorithm into two parts: an approximation method for problem (\ref{Rv}) and an augmented Lagrangian (AL) method for its subproblem (\ref{Rvsig}) with a fixed parameter $\sigma$.

\subsection{An approximation method for problem (\ref{Rv})}
In Section 1, we have mentioned that problem (\ref{Rvsig}) is an approximation problem of problem (\ref{Rv}). Since it is difficult to find an exact optimal solution of problem (\ref{Rvsig}), in this subsection, we present a method in which we only need to find an approximate stationary point of problem ($R_{\nu,\sigma}$) at each iteration and prove the convergence of this method to an MPCC lifted-stationary point of problem (\ref{Rv}).
\begin{algorithm}[H]
\caption{ Approximation method}
\label{alg:A}
\begin{algorithmic}
\normalsize
\STATE {Let $\{\sigma_k\}$ and $\{\epsilon^k\}$ be sequences of nonnegative parameters approaching to 0. Choose an arbitrary point $x^0\in\mathbb{R}^n$ and set $k=1$.\\
(1) Solve problem ($R_{\nu,\sigma_k})$ with initial point $x^{k-1}$ to get $x^k$ such that there exist
$d = (d_1, \cdots,d_n)^T$ with $d_i\in {\cal D}(x_i^k), i=1,\ldots, n$ and
$\zeta^k\in \mathbb{R}^n $ with $\|\zeta^k\|\le \epsilon^k$
satisfying
\begin{equation}\label{stopping}
  \zeta^k\in\partial\left(\sum_{i=1}^{n}\frac{|x^k_i|}{\nu}\right)-\sum_{i=1}^{n}\theta'_{d_i}(x^k_i)\boldsymbol{e}_i+\mathcal{N}_{\mathcal{S}_{\sigma_k}}(x^k).
\end{equation}
(2) Set $k\leftarrow k+1$, and go to Step (1).}\\
\textbf{END}
\end{algorithmic}
\end{algorithm}
 \begin{theorem}\label{algAMPCClifted}
    Let $\{x^k\}$ be a sequence generated by Algorithm \ref{alg:A} and $x^*$ be an arbitrary accumulation point of $\{x^k\}$. Suppose that MPCC-LICQ holds at $x^*$. Then $x^*$ is an MPCC lifted-stationary point of problem (\ref{Rv}).
  \end{theorem}
\begin{proof}
By Proposition \ref{3.1} and relations that $\sigma_k\rightarrow0$ and $x^k\rightarrow x^*$, it follows that LICQ holds at $x^k\in\mathcal{S}_{\sigma_k}$ when $k$ is sufficiently large. This together with Theorem 6.14 in \cite{s4}, implies that
      $$ \mathcal{N}_{\mathcal{S}_{\sigma_k}}(x^k)\subseteq\left\{
\begin{aligned}
 &-\nabla G(x^k)\alpha^k-\nabla H(x^k)\beta^k+&   \alpha^k\geq0, G(x^k)^T\alpha^k=0 \\
&[\nabla G(x^k)H(x^k)+\nabla H(x^k)G(x^k)]\gamma^k:\   & \beta^k\geq0, H(x^k)^T\beta^k=0 \\
& &  \gamma^k\geq0, h_{\sigma_k}(x^k)\gamma^k=0
\end{aligned}
\right\}.
$$
    It then follows from (\ref{stopping}) that for all $k$ sufficiently large, there exist $\alpha^k,\beta^k,\gamma^k$ and $d^k=(d^k_1,\cdots,d^k_n)^T$ with $d^k_i\in\mathcal{D}(x^k_i)$, $i=1,\ldots,n$, such that
    \begin{equation}\label{4.3}
    \begin{split}
       \zeta^k & \in\partial\left(\sum_{i=1}^{n}\frac{|x^k_i|}{\nu}\right)-\sum_{i=1}^{n}\theta'_{d^k_i}(x^k_i)\boldsymbol{e}_i \\
         & -\nabla G(x^k)(\alpha^k-\gamma^kH(x^k))-\nabla H(x^k)(\beta^k-\gamma^kG(x^k)),
    \end{split}
    \end{equation}
    \begin{equation}\label{4.4}
      \alpha^k\geq0, G(x^k)^T\alpha^k=0, \beta^k\geq0, H(x^k)^T\beta^k=0, \gamma^k\geq0, h_{\sigma_k}(x^k)\gamma^k=0.
    \end{equation}
For simplicity, we denote $\mathcal{I}^*_{0+}:=\mathcal{I}_{0+}(x^*)$, $\mathcal{I}^*_{00}:=\mathcal{I}_{00}(x^*)$, and $\mathcal{I}^*_{+0}:=\mathcal{I}_{+0}(x^*)$. Similar to the proof of Theorem 4.1 in \cite{s44}, one has
    \begin{equation}\label{4.9}
         \sum_{i\in\mathcal{I}^*_{0+}\cup\mathcal{I}^*_{00}}u_i^k\nabla G_i(x^k)+\sum_{i\in\mathcal{I}^*_{+0}\cup\mathcal{I}^*_{00}}v_i^k\nabla H_i(x^k)\in-\zeta^k+\partial\left(\sum_{i=1}^{n}\frac{|x^k_i|}{\nu}\right)-\sum_{i=1}^{n}\theta'_{d^k_i}(x^k_i)\boldsymbol{e}_i,
    \end{equation}
    where
    \begin{equation}\label{4.6}
      u^k:=\alpha^k-\gamma^kH(x^k), v^k:=\beta^k-\gamma^kG(x^k).
    \end{equation}
    Due to the boundedness of the right hand side of (\ref{4.9}), by MPCC-LICQ at $x^*$, it is not hard to verify that there exist subsequences $\{u_i^{k_j}:i\in\mathcal{I}^*_{0+}\cup\mathcal{I}^*_{00}\}$ and $\{v_i^{k_j}:i\in\mathcal{I}^*_{+0}\cup\mathcal{I}^*_{00}\}$ of sequences $\{u_i^{k}:i\in\mathcal{I}^*_{0+}\cup\mathcal{I}^*_{00}\}$ and $\{v_i^{k}:i\in\mathcal{I}^*_{+0}\cup\mathcal{I}^*_{00}\}$, respectively, converge to some $\{u_i^*:i\in\mathcal{I}^*_{0+}\cup\mathcal{I}^*_{00}\}$ and $\{v_i^*:i\in\mathcal{I}^*_{0+}\cup\mathcal{I}^*_{00}\}$. Taking limits on both sides of (\ref{4.9}), since the elements in $\mathcal{D}(x^*_i)$ for $ i=1,\ldots,n$ are finite, we have that there exists $d^*=(d^*_1,\cdots,d^*_n)^T$ with $d^*_i\in\mathcal{D}(x^*_i)$, $i=1,\ldots,n$, such that
    \begin{equation}\label{4.10}
      \sum_{i\in\mathcal{I}^*_{0+}\cup\mathcal{I}^*_{00}}u_i^*\nabla G_i(x^*)+\sum_{i\in\mathcal{I}^*_{+0}\cup\mathcal{I}^*_{00}}v_i^*\nabla H_i(x^*)\in\partial\left(\sum_{i=1}^{n}\frac{|x^*_i|}{\nu}\right)-\sum_{i=1}^{n}\theta'_{d^*_i}(x^*_i)\boldsymbol{e}_i.
    \end{equation}
    The inequality $u^*_iv^*_i\geq0$ follows immediately since
    \begin{equation*}
      \begin{split}
         u^{k_j}_iv^{k_j}_i & =(\alpha^{k_j}_i-\gamma^{k_j}H_i(x^{k_j}))(\beta^{k_j}_i-\gamma^{k_j}G_i(x^{k_j})) \\
           & =\alpha^{k_j}_i\beta^{k_j}_i-\gamma^{k_j}\alpha^{k_j}_iG_i(x^{k_j})-\gamma^{k_j}\beta^{k_j}_iH_i(x^{k_j})+(\gamma^{k_j})^2G_i(x^{k_j})H_i(x^{k_j}) \\
           & =\alpha^{k_j}_i\beta^{k_j}_i+(\gamma^{k_j})^2G_i(x^{k_j})H_i(x^{k_j})\geq0,
      \end{split}
    \end{equation*}
    where the last equality follows from (\ref{4.4}). This together with (\ref{4.10}) implies that $x^*$ is an MPCC lifted-stationary point of problem (\ref{Rv}).\qed
\end{proof}

\subsection{An AL method for problem (\ref{Rvsig})}\label{subAL}

\par In this subsection, for each fixed $\sigma>0$, we propose an AL method to solve problem (\ref{Rvsig}). We handle the easy constraints
$g(x)\le 0$ directly while penalize the hard constraint $h_\sigma(x)\le 0$  into the
objective function.
Therefore, problem (\ref{Rvsig}) can be rewritten as
\begin{equation}\tag{$R$\textsuperscript{AL}}\label{RvsigAL}
\begin{split}
\min_{x\in\mathcal{F}}\ & \Phi(x)\\
{\rm s.t.}\ & h_{\sigma}(x)\leq0.
\end{split}
\end{equation}
\par For any given penalty parameter $\rho>0$ and Lagrangian multipliers $\mu$, the AL function for problem (\ref{RvsigAL}) is defined as
\begin{equation}\label{Lfun}
  \mathcal{L}(x,\mu,\rho):=\Phi(x)+\varphi(x,\mu,\rho),
\end{equation}
where $\varphi(x,\mu,\rho):=\frac{1}{2\rho}([\mu+\rho h_{\sigma}(x)]_+^2-\mu^2)$.
\par  At each outer iteration we approximate problem (\ref{RvsigAL}) by an AL subproblem with fixed $\rho>0$ and $\mu$ in the form of
\begin{equation}\tag{$P$\textsuperscript{AL}}\label{PAL}
\begin{split}
\min_{x\in\mathcal{F}}\ & \mathcal{L}(x,\mu,\rho).
\end{split}
\end{equation}
\begin{definition}
  \cite{s68} We say that $x^*\in\mathcal{S}_{\sigma}$ is a lifted stationary point of problem (\ref{Rvsig}) if $\mathcal{S}_{\sigma}$ is regular at $x^*$ and there exists $d=(d_1,\cdots,d_n)^T$ with $d_i\in\mathcal{D}(x^*_i)$, $i=1,\ldots,n$, such that
  \begin{equation}\label{Rvsiglifted}
0\in\partial\left(\sum_{i=1}^{n}\frac{|x^*_i|}{\nu}\right)-\sum_{i=1}^{n}\theta'_{d_i}(x^*_i)\boldsymbol{e}_i+\mathcal{N}_{\mathcal{S}_{\sigma}}(x^*).
  \end{equation}
\end{definition}
\par From \cite{s31}, it follows that lifted stationarity is stronger than Clarke stationarity. Denote a feasible point of problem (\ref{RvsigAL}) by $x^{feas}_{\sigma}$. For any fixed $\sigma$, the AL method for problem (\ref{Rvsig}) is proposed as follows.

\begin{algorithm}[H]
\caption{ AL method}
\label{alg:B}
\begin{algorithmic}
\normalsize
\STATE {Choose $\mu^0\in\mathbb{R}$, $x^0_{init}\in\mathcal{F}$, $\rho_0>0$, $\gamma\in(1,\infty)$, $\tau,\vartheta\in(0,1)$, a nonnegative sequence $\{\epsilon_l\}$, and a constant}
\begin{equation*}
  \Upsilon\geq \max\{\Phi(x^{feas}_{\sigma}),\mathcal{L}(x^0_{init},\mu^0,\rho_0)\}.
\end{equation*}
\STATE {Set $l=0$.
\par(1) Solve problem ($P^{AL}$) with $\mu=\mu^l$ and $\rho=\rho_l$ to find a point $x^l\in\mathcal{F}$ such that there exists $d=(d_1,\cdots,d_n)^T$ with $d_i\in\mathcal{D}(x^l_i)$, $i=1,\ldots,n$, satisfying
\begin{equation}\label{stopping2}
\begin{split}
    &  {\rm dist}(0,\partial\varphi(x^l,\mu^l,\rho_l)+\partial\left(\sum_{i=1}^{n}\frac{|x^l_i|}{\nu}\right)-\sum_{i=1}^{n}\theta'_{d_i}(x^l_i)\boldsymbol{e}_i+\mathcal{N}_{\mathcal{F}}(x^l))\leq\epsilon_l, \\
     & \mathcal{L}(x^l,\mu^l,\rho_l)\leq\Upsilon.
\end{split}
  \end{equation}
\par(2) Set
\begin{gather}
  \mu^{l+1}=[\mu^l+\rho_lh_{\sigma}(x^l)]_+, \label{Upmu}\\
   \xi^{l+1}=\min\{\mu^{l+1}\/\rho_l,-h_{\sigma}(x^l)\},\label{Upxi}
\end{gather}
\par(3) If $l>0$ and
\begin{equation}\label{Conditionrho}
  |\xi^{l+1}|\leq\vartheta|\xi^l|,
\end{equation}
then set $\rho_{l+1}=\rho_l$. Otherwise, set
\begin{equation}\label{Uprho}
  \rho_{l+1}=\max\{\gamma\rho_l,(\mu^{l+1})^{1+\tau}\}.
\end{equation}
(4) Set $l\leftarrow l+1$ and go to Step (1).}\\
\textbf{END}
\end{algorithmic}
\end{algorithm}
\begin{theorem}\label{AL}
   Assume that $\lim_{l\rightarrow\infty}\epsilon_l=0$ for Algorithm \ref{alg:B}. Let $\{x^l\}$ be the sequence generated by Algorithm \ref{alg:B} and $x^*$ an accumulation point of $\{x^l\}$. Then the following statements hold:
     \par (i) $[h_{\sigma}(x^l)]_+\rightarrow0$ as $l\rightarrow\infty$.
  \par (ii) $x^*$ is a feasible point of problem (\ref{Rvsig}).
  \par (iii) Assume that $\mathcal{S}_{\sigma}$ is regular at $x^*$. If RCPLD holds at $x^*$ for system $g(x)\leq0$, then $x^*$ is a lifted stationary point of problem (\ref{Rvsig}).
\end{theorem}

\begin{proof}
  (i) We prove the statement (i) by considering two cases.
    \par (a) $\{\rho_l\}$ is bounded. According to update rule, one can know that (\ref{Uprho}) is updated for finite times. It implies that there exists $k_0$ such that for all $k\geq k_0$, (\ref{Conditionrho}) holds. Thus, $\lim_{l\rightarrow\infty}|\xi^l|=0$. By (\ref{Upxi}), one can see that $\xi^{l+1}\leq -h_{\sigma}(x^l)$ and thus $(h_{\sigma}(x^l))_+\leq(-\xi^{l+1})_+\leq|\xi^{l+1}|$, which yields that statement (i) holds.
  \par (b) $\{\rho_l\}$ is unbounded. Then, it follows from the proof of Theorem 3.1, (i) in \cite{s64} that
    \begin{equation}\label{murho0}
    \lim_{l\rightarrow\infty}\mu^l/\rho_l=0.
  \end{equation}
  \par Further, by the second relation in (\ref{stopping2}) and the definition of the AL function (\ref{Lfun}), we can show that
    \begin{equation*}
    \left[\frac{\mu^l}{\rho_l}+h_{\sigma}(x^l)\right]_+^2\leq\frac{2}{\rho_l}(\Upsilon-\Phi(x^l))+\left(\frac{\mu^l}{\rho_l}\right)^2,
  \end{equation*}
  which together with (\ref{murho0}), unbounded $\{\rho_l\}$ and the lower boundedness of $\{\Phi(x^l)\}$ implies that $[h_{\sigma}(x^l)]_+\rightarrow0$ as $l\rightarrow\infty$. This completes the proof of statement (i).
  \par (ii) Since $x^*$ is an accumulation point of $\{x^l\}$, it follows from the continuity of function $h_{\sigma}(x)$ and statement (i) that $[h_{\sigma}(x^*)]_+=0$, which yields $h_{\sigma}(x^*)\leq0$. Recall that $x^*\in\mathcal{F}$. Hence, $x^*$ is a feasible point of problem (\ref{Rvsig}).
  \par (iii) For any fixed $\sigma>0$, $x^*$ is an accumulation point of $\{x^l\}$, there exists a subsequence $\mathcal{K}$ such that $\{x^l\}_{\mathcal{K}}\rightarrow x^*$. Let $\mathcal{H}:=\{u\in\mathbb{R}^n: h_{\sigma}(u)=0\}$ and $\mathcal{H}^{-}:=\{u\in\mathbb{R}^n: h_{\sigma}(u)<0\}$. Next, we prove
  \begin{equation}\label{mu0sub}
    \{\textbf{1}_{\mathcal{H}^{-}}(x^*)\mu^{l+1}\}_{\mathcal{K}}\rightarrow0.
  \end{equation}
  From (ii), we know that $h_{\sigma}(x^*)\leq0$. If $h_{\sigma}(x^*)=0$, then (\ref{mu0sub}) holds automatically. Thus, in the following two cases, we only consider $h_{\sigma}(x^*)<0$.
  \par Case (a): $\{\rho_l\}$ is bounded. It is easy to know that $\xi^{l+1}\rightarrow0$ as $k\rightarrow\infty$. Since $\{x^l\}_{\mathcal{K}}\rightarrow x^*$, we have that $\textbf{1}_{\mathcal{H}^{-}}(x^*)h_{\sigma}(x^l)<\textbf{1}_{\mathcal{H}^{-}}(x^*)h_{\sigma}(x^*)/2<0$ for sufficiently large $l\in\mathcal{K}$. It follows from this and (\ref{Upxi}) that $\{\mu^{l+1}/\rho_l\}_{\mathcal{K}}\rightarrow0$, which together with the boundedness of $\{\rho_l\}$ yields (\ref{mu0sub}).
  \par Case (b): $\{\rho_l\}$ is unbounded. Recall from the proof of statement (i) that $\mu^l/\rho_l\rightarrow0$. From above, we know that $\textbf{1}_{\mathcal{H}^{-}}(x^*)h_{\sigma}(x^l)<\textbf{1}_{\mathcal{H}^{-}}(x^*)h_{\sigma}(x^*)/2<0$ for sufficiently large $l\in\mathcal{K}$. It follows from these and the relation of (\ref{Upmu}) that for sufficiently large $l\in\mathcal{K}$,
  \begin{equation*}
    \textbf{1}_{\mathcal{H}^{-}}(x^*)\mu^{l+1}=\textbf{1}_{\mathcal{H}^{-}}(x^*)\rho_l[\mu^l/\rho_l+h_{\sigma}(x^l)]_+=0,
  \end{equation*}
  and hence (\ref{mu0sub}) holds. For convenience, let
  \begin{equation*}
    \mathcal{I}_{g}(x):=\{i: g_i(x)=0\},\ \forall x\in\mathbb{R}^n.
  \end{equation*}
\par Since RCPLD holds at $x^*$ for system $g(x)\leq0$, it follows from Proposition 2.1 and 2.2 in \cite{s64} that there exists $\delta>0$ such that for any $x\in\mathcal{B}_{\delta}(x^*)$,
  \begin{equation}\label{ConCQ1}
       \mathcal{N}_{\mathcal{F}}(x)=\left\{\sum_{i\in\mathcal{I}_g(x)}\beta_i\nabla g_i(x): \beta_i\geq0, i\in\mathcal{I}_g(x)\right\}.
  \end{equation}
  Let $\mathcal{I}_g^l:=\mathcal{I}_g(x^l)$. By the definition of $\varphi$, (\ref{ConCQ1}), the first relation in (\ref{stopping2}) and Step (2) in Algorithm \ref{alg:B}, there exist $\beta^{l+1}\in\mathbb{R}^{|\mathcal{I}_g^l|}_+$, $\eta^l\in\mathbb{R}^n$ and $d^l=(d^l_1,\cdots,d^l_n)^T$ with $d^l_i\in\mathcal{D}(x^l_i)$, $i=1,\ldots,n$ such that
  \begin{equation}\label{ConCQ2}
\eta^l\in\partial\left(\sum_{i=1}^{n}\frac{|x^l_i|}{\nu}\right)-\sum_{i=1}^{n}\theta'_{d^l_i}(x^l_i)\boldsymbol{e}_i+\mu^{l+1}\nabla h_{\sigma}(x^l)+\sum_{i\in\mathcal{I}_g^l}\beta_i^{l+1}\nabla g(x^l)
  \end{equation}
  and $\|\eta^l\|\leq\epsilon_l$ for all $l$. It then follows from $\epsilon_l\rightarrow0$ that $\eta^l\rightarrow0$.
  \par Thus it follows from (\ref{ConCQ2}) that for every sufficiently large $l\in\mathcal{K}$, we have
  \begin{equation}\label{ConCQ3}
\tilde{\eta}^l\in\partial\left(\sum_{i=1}^{n}\frac{|x^l_i|}{\nu}\right)-\sum_{i=1}^{n}\theta'_{d^l_i}(x^l_i)\boldsymbol{e}_i+\mu^{l+1}\textbf{1}_{\mathcal{H}}(x^*)\nabla h_{\sigma}(x^l)+\sum_{i\in\mathcal{I}_g^l}\beta_i^{k+1}\nabla g_i(x^l),
  \end{equation}
  where $\tilde{\eta}^l:=\eta^l-\textbf{1}_{\mathcal{H}^{-}}(x^*)\mu^{l+1}\nabla h_{\sigma}(x^l)$. Taking limits on both sides of (\ref{ConCQ3}) as $\mathcal{K}\ni l\rightarrow\infty$, since the elements in $\mathcal{D}(x^*_i)$ for $i=1,\ldots,n$ are finite, it follows from Proposition 2.2 in \cite{s64} and the proof of Theorem 3.1, (iii) in \cite{s64} that there exist $\mu^*\in\mathbb{R}$, $\beta^*_i\geq0, i\in\mathcal{I}^*_g$ and $d^*=(d^*_1,\cdots,d^*_n)^T$ with $d^*_i\in\mathcal{D}(x^*_i)$, $i=1,\ldots,n$, such that
  \begin{equation}\label{liftedRvsig}
  \begin{split}
     0 & \in\partial\left(\sum_{i=1}^{n}\frac{|x^*_i|}{\nu}\right)-\sum_{i=1}^{n}\theta'_{d^*_i}(x^*_i)\boldsymbol{e}_i+\textbf{1}_{\mathcal{H}}(x^*)\mu^*\nabla h_{\sigma}(x^*)+\sum_{i\in\mathcal{I}_g^*}\beta_i^*\nabla g_i(x^*) \\
& \subseteq\partial\left(\sum_{i=1}^{n}\frac{|x^*_i|}{\nu}\right)-\sum_{i=1}^{n}\theta'_{d^*_i}(x^*_i)\boldsymbol{e}_i+\mathcal{N}_{\mathcal{S}_{\sigma}}(x^*).
  \end{split}
\end{equation}
  Together with the regularity of $\mathcal{S}_{\sigma}$ at $x^*$, (\ref{liftedRvsig}) implies that $x^*$ is a lifted stationary point of problem (\ref{Rvsig}).\qed
\end{proof}

In general, it is  difficult to find a point $x^l \in {\cal F} $  satisfying  (\ref{stopping2}) in  Algorithm \ref{alg:B}, due to the nonconvexity of problem ($P^{AL}$). For some special cases, we can use  some existing augmented Lagrangian algorithms to find such approximate stationary points. For example, see \cite{s64} and its references. In the next section, we will consider a special case  where functions $G$ and $H$ are affine functions in Proposition \ref{PSDset}.  We propose a DC algorithm to solve the AL subproblem (\ref{PAL}) efficiently. In particular, we can find a point satisfying (\ref{stopping2}) by solving strongly monotone subproblem (\ref{DCset}) in finite steps.

\section{Sparse solutions of VLCS}

Given matrices $A,C\in\mathbb{R}^{m\times n}$ and vectors $b,d\in\mathbb{R}^m$, a vertical linear complementarity system (VLCS) is to find $x\in\mathbb{R}^n$ such that
\begin{equation}\label{VLCP}
 Ax-b\geq0, Cx-d\geq0, (Ax-b)^T(Cx-d)=0.
\end{equation}
In this section, we consider the case where $G(x)=Ax-b$ and $H(x)=Cx-d$ which means $\mathcal{F}=\{x\in\mathbb{R}^n: Ax-b\geq0, Cx-d\geq0\}$ and
 $ \mathcal{S}=\{x\in\mathbb{R}^n:0\leq Ax-b\bot Cx-d\geq0\}.$
We consider the following sparse optimization problem:
\begin{equation}\label{NEP0}
  \begin{split}
     \min &\ \|x\|_0 \\
      {\rm s.t.} &\ 0\leq Ax-b\bot Cx-d\geq0.
  \end{split}
\end{equation}
In subsection 5.1, we give sufficient conditions for Assumption 1 to hold, and in subsection 5.2, we apply Algorithms \ref{alg:A} and \ref{alg:B} to find a sparse solution of VLCS. In subsection 5.3, we show that a class of  equilibrium problems in market modeling can be formulated as problem (\ref{NEP0}) and in subsection 5.4, we present numerical results for solving problem (\ref{NEP0}).

\subsection{Solvability of problem (\ref{NEP0})}

\par We say that problem (\ref{VLCP}) is feasible, if there exists an $\bar{x}\in\mathcal{F}$. And we say that problem (\ref{VLCP}) is solvable, if $\mathcal{S}$ is not empty. A matrix $M\in\mathbb{R}^{n\times n}$ is called column sufficient if $z_i(Mz)_i\leq0$ for all $i=1,2,\ldots,n$ implies $z_i(Mz)_i=0$ for all $i=1,2,\ldots,n$. The matrix $M$ is called row sufficient if its transpose is column sufficient. A matrix $M\in\mathbb{R}^{n\times n}$ is called copositive if $z^TMz\geq0$ for all $z\in\mathbb{R}^n_+$; and a matrix $M\in\mathbb{R}^{n\times n}$ is called copositive-plus if $M$ is copositive and $z^TMz=0, z\geq0$ imply $(M+M^T)z=0$.
\begin{lemma}\label{FeaSol}
  Assume that matrix $C$ has full row rank. Let $M:=AC^T(CC^T)^{-1}$.
  \par (i) If $M$ is row sufficient, and problem (\ref{VLCP}) is feasible, then it is solvable.
  \par (ii) If $M$ is copositive-plus and $$ \{z\in\mathbb{R}^m:(M+M^T)z=0\}\subseteq\{z\in\mathbb{R}^m:(Md-b)^Tz=0\}.$$ Then, problem (\ref{VLCP}) is solvable.
\end{lemma}
\begin{proof}
  (i) Let $x_0\in\mathcal{F}$. Consider the following LCP: find $y\in\mathbb{R}^m$ such that
  \begin{equation}\label{ProofFeaLCP}
0\leq y\bot My+Md-MCx_0+Ax_0-b\geq0.
\end{equation}
Then, it is easy to verify that $Cx_0-d$ is a feasible vector of LCP (\ref{ProofFeaLCP}). Since matrix $M$ is row sufficient, due to Theorem 3.5.4 in \cite{s1}, we know the LCP (\ref{ProofFeaLCP}) is solvable, which implies that there exists a vector $y_0\in\mathbb{R}^m$ satisfying (\ref{ProofFeaLCP}).
\par Let $\bar{x}=C^T(CC^T)^{-1}(y_0+d)+(I_n-C^T(CC^T)^{-1}C)x_0$. Then, we have
\begin{equation*}
  \begin{split}
      & C\bar{x}-d=y_0, \\
      & A\bar{x}-b=My_0+Md-MCx_0+Ax_0-b,
  \end{split}
\end{equation*}
which imply that $\bar{x}$ is solution of problem (\ref{VLCP}).
\par (ii) Consider the following LCP: find $y\in\mathbb{R}^m$ such that
  \begin{equation}\label{ProofFeaLCPyii}
0\leq y\bot My+Md-b\geq0.
\end{equation}
\par Define function $f(y):=y^T(My+Md-b), y\in\mathbb{R}^m$. Consider a set $\mathcal{W}=(\mathbb{R}^n_+\setminus\{z\in\mathbb{R}^m:(M+M^T)z=0\})\cup\{0\}$. Then,
\begin{equation*}
  \begin{split}
     &\inf_{y\in\mathbb{R}^m_+} f(y)\\
      =& \inf_{y\in\mathcal{W}} f(y) =\inf_{y\in\mathcal{W}} \frac{1}{2}y^T(M^T+M)y+y^T(Md-b) \\
       =&\inf_{y\in\mathcal{W},\|y\|=1,\lambda\geq0} \frac{1}{2}\lambda^2y^T(M^T+M)y+\lambda y^T(Md-b) \\
      \geq & \inf_{\lambda\geq0} \left\{ \lambda^2 \left\{ \min_{y\in\mathcal{W},\|y\|=1}\frac{1}{2}y^T(M^T+M)y\right\}+\lambda\left\{ \min_{y\in\mathcal{W},\|y\|=1} y^T(Md-b)\right\} \right\},\\
  \end{split}
\end{equation*}
where the first equality holds because $0=f(0)\leq f(y)$ when $M$ is copositive-plus and $y\in \{z\in\mathbb{R}^m:(M+M^T)z=0\}\subseteq\{z\in\mathbb{R}^m:(Md-b)^Tz=0\}.$
\par Since matrix $M$ is copositive-plus and $\mathcal{W}\cap\{z\in\mathbb{R}^m:(M+M^T)z=0\}=\{0\}$, the quantity
\begin{equation*}
  \alpha:=\min_{y\in\mathcal{W},\|y\|=1}\frac{1}{2}y^T(M^T+M)y
\end{equation*} is strictly positive. Let $\beta:=\min_{y\in\mathcal{W},\|y\|=1} y^T(Md-b)$. Then, we know
\begin{equation*}
  \inf_{y\in\mathbb{R}^m_+}f(y)\geq\inf_{\lambda\in\mathbb{R}}\alpha\lambda^2+\beta\lambda=-\frac{\beta^2}{4\alpha},
\end{equation*}
which implies that $f$ is bounded below for $y\geq0$. Then according to Corollary 3.7.12 in \cite{s1}, we know that the LCP (\ref{ProofFeaLCPyii}) has a solution and denote a solution by $\bar{y}$.
Let $\tilde{x}=C^T(CC^T)^{-1}(\bar{y}+d)$. From
\begin{equation*}
  \begin{split}
      & C\tilde{x}-d=\bar{y}\geq0, \\
      & A\tilde{x}-b=M\bar{y}+Md-b\geq0, \\
      & (A\tilde{x}-b)^T(C\tilde{x}-d)=\bar{y}^T(M\bar{y}+Md-b)=0,
  \end{split}
\end{equation*}
we know that $\tilde{x}$ is a solution of problem (\ref{VLCP}). Thus, under the assumptions in (ii), we obtain that problem (\ref{VLCP}) has a solution.\qed
\end{proof}
\par It follows from Theorem 6.9 in \cite{s4} that the set $\mathcal{S}$ is regular at any feasible point when $\mathcal{S}$ is convex. Since regularity is required in some previous results, here we give a sufficient condition for the convexity of set $\mathcal{S}$.
\begin{proposition}\label{PSDset}
  If $A^TC$ is positive semi-definite and problem (\ref{VLCP}) has a solution $x^*$, then the set $\mathcal{S}$ equals to
    $$ \bar{\mathcal{S}}:=\left\{
\begin{aligned}
 &  Ax-b\geq0,\ Cx-d\geq0,  \\
x\in\mathbb{R}^n:\   &(A^TC+C^TA)(x-x^*)=0,\\
&(b^TC+d^TA)(x-x^*)=0
\end{aligned}
\right\}.
$$
\end{proposition}
\begin{proof}
  Let $z^1$ and $z^2$ be arbitrary two solutions of problem (\ref{VLCP}). Since matrix $A^TC$ is positive semi-definite, we have
  \begin{equation*}
    \begin{split}
       0 & \leq(z^1-z^2)^TA^TC(z^1-z^2) \\
         & =(Az^1-b-(Az^2-b))^T(Cz^1-d-(Cz^2-d)) \\
         & =-(Az^1-b)^T(Cz^2-d)-(Az^2-b)^T(Cz^1-d) \\
         & \leq0.
    \end{split}
  \end{equation*}
  Thus, for any two solutions $z^1$ and $z^2$, we have
  \begin{equation}\label{wuniq}
    (Az^1-b)^T(Cz^2-d)=(Az^2-b)^T(Cz^1-d)=0,
  \end{equation}
  and
  \begin{equation}\label{AC0}
    (z^1-z^2)^TA^TC(z^1-z^2)=0.
  \end{equation}
  \par Let $z$ be an arbitrary solution. By (\ref{AC0}), we know that $0=(z-x^*)^TA^TC(z-x^*)=\frac{1}{2}(z-x^*)^T(A^TC+C^TA)(z-x^*)$. Since $A^TC$ is positive semi-definite, the gradient of the quadratic function satisfies $(A^TC+C^TA)(z-x^*)=0$. Thus, we have
  \begin{equation*}
    \begin{split}
        & z^T(A^TC+C^TA)z=z^T(A^TC+C^TA)x^*, \\
        & x^{*T}(A^TC+C^TA)x^*=x^{*T}(A^TC+C^TA)z.
    \end{split}
  \end{equation*}
  The above two equalities imply that $z^TA^TCz=x^{*T}A^TCx^*$. Meanwhile, we know
  \begin{equation*}
    0=(Az-b)^T(Cz-d)=(Ax^*-b)^T(Cx^*-d).
  \end{equation*}
  Hence, $(b^TC+d^TA)(z-x^*)=0$ and $z\in\bar{\mathcal{S}}$.
  \par Conversely, suppose that $z\in\bar{\mathcal{S}}$. From $(A^TC+C^TA)(z-x^*)=0$, it follows that $z^TA^TCz=x^{*T}A^TCx^*$ just by the argument we used above. Since $(b^TC+d^TA)(z-x^*)=0$, we have
  \begin{equation*}
    (Az-b)^T(Cz-d)=(Ax^*-b)^T(Cx^*-d)=0.
  \end{equation*}
The proof is completed.\qed
\end{proof}

\begin{remark}
 Under assumptions of Proposition \ref{PSDset},  the feasible set $\mathcal{S}$ of problem (\ref{NEP0}) is convex, which implies $\mathcal{S}$ is regular at any point $x\in\mathcal{S}$. Moreover, if $A^TC$ is positive definite, then there exists a unique solution of problem (\ref{VLCP}).
\end{remark}

\subsection{DC algorithm for subproblem (\ref{PAL})}
In this subsection, we discuss how to find an approximate stationary point $x^l$ of the $l$th AL subproblem (\ref{PAL}) satisfying (\ref{stopping2}) as required in Step (1) of Algorithm \ref{alg:B} for sparse optimization problem (\ref{NEP0}).
\par Under assumptions of Proposition \ref{PSDset}, we know that $\varphi(\cdot,\mu,\rho)$ is a convex function with respect to $x$. Recall that $\mathcal{F}=\{x\in\mathbb{R}^n: Ax-b\geq0, Cx-d\geq0\}$ for VLCS and  $\theta_1(t)=0$, $\theta_2(t)=t/\nu-1$ and $\theta_3(t)=-t/\nu-1$ are linear functions. Thus, problem (\ref{PAL}) with fixed parameters $\mu$ and $\rho$ can be rewritten as the following DC programming problem:
\begin{equation}\label{PALDC}
\begin{split}
\min_{x\in\mathcal{F}}\ & f_1(x)-f_2(x),
\end{split}
\end{equation}
where $f_1(x)=\sum_{i=1}^{n}\frac{|x_i|}{\nu}+\varphi(x,\mu,\rho)$, $f_2(x)=\sum_{i=1}^{n}\max\{\theta_1(x_i),\theta_2(x_i),\theta_3(x_i)\}$ are two convex functions. Define function $\Psi(x)=f_1(x)-f_2(x), x\in\mathbb{R}^n$. It is easy to see that $f_2$ is directionally differentiable.
\par For any $x\in\mathbb{R}^n$, let
$$ d^x_i=\left\{
             \begin{array}{rcl}
             &1, & \ |x_i|<\nu,  \\
             &2, & \ x_i\geq\nu,\\
             &3,& \ x_i\leq-\nu.
             \end{array}
\right.$$
\par We now propose a DC algorithm \cite{RecentDC} to find an approximate stationary point of AL subproblem (\ref{PAL}) satisfying (\ref{stopping2}).
\begin{algorithm}[H]
\caption{DC Algorithm}
\label{alg:C}
\begin{algorithmic}
\normalsize
\STATE {Choose $x^0\in\mathcal{F}$ and set $\jmath = 0$.
\par (1) Let $d^{\jmath}:=d^{x^{\jmath}}$, and
\begin{equation}\label{DCset}
  \begin{split}
       x^{\jmath+1}=&\argmin_{y\in\mathcal{F}}  f_1(y)-\sum_{i=1}^n\theta_{d^{\jmath}_i}(x^{\jmath}_i)-\sum_{i=1}^{n}\theta'_{d^{\jmath}_i}(x^{\jmath}_i)(y_i-x^{\jmath}_i)+\frac{1}{2}\|y-x^{\jmath}\|^2.
  \end{split}
\end{equation}
(2) Set ${\jmath}\leftarrow {\jmath}+1$ and go to Step (1).}\\
\textbf{END}
\end{algorithmic}
\end{algorithm}
Note that in  subproblem (\ref{DCset}), the objective function is strongly convex and the constraints are linear. Hence there exists a unique solution to subproblem (\ref{DCset}) and Algorithm \ref{alg:C} is well-defined under assumptions of Proposition \ref{PSDset}.
\begin{theorem}
  Suppose that $x^0\in\mathcal{F}$ and $L(x^0):=\{x\in\mathcal{F}:\Psi(x)\leq\Psi(x^0)\}$ is bounded. Then the sequence $\{x^{\jmath}\}$ generated by Algorithm \ref{alg:C} with starting point $x^0$ is bounded and any accumulation point is a lifted stationary point of problem (\ref{PALDC}).
\end{theorem}
\begin{proof}
   By the update rule of the algorithm, we have
  \begin{equation*}
    \begin{split}
       \Psi(x^{\jmath}) & = f_1(x^{\jmath})-\sum_{i=1}^{n}\theta_{d^{\jmath}_i}(x^{\jmath}_i)\\
         & \geq f_1(x^{\jmath+1})-\sum_{i=1}^{n}(\theta_{d^{\jmath}_i}(x^{\jmath}_i)+\theta'_{d^{\jmath}_i}(x^{\jmath}_i)(x^{\jmath+1}_i-x^{\jmath}))+\frac{1}{2}\|x^{\jmath+1}-x^{\jmath}\|^2 \\
         & \geq f_1(x^{\jmath+1})-\sum_{i=1}^{n}\theta_{d^{\jmath}_i}(x^{\jmath+1}_i)+\frac{1}{2}\|x^{\jmath+1}-x^{\jmath}\|^2\\
         &\geq f_1(x^{\jmath+1})-f_2(x^{\jmath+1})+\frac{1}{2}\|x^{\jmath+1}-x^{\jmath}\|^2\\
         &= \Psi(x^{\jmath+1})+\frac{1}{2}\|x^{\jmath+1}-x^{\jmath}\|^2,
    \end{split}
  \end{equation*}
  where the first inequality holds by the definition of $x^{\jmath+1}$, the second inequality holds by the convexity of $\theta_j, j=1,2,3$, the third inequality holds by the definition of $f_2(\cdot)$ and the last equality holds by the definition of $\Psi(\cdot)$. Hence, the sequence of objective values $\{\Psi(x^{\jmath})\}$ is nonincreasing, and strictly decreasing if $x^{\jmath+1}\neq x^{\jmath}$ for all $\jmath$. Since $\Psi(\cdot)$ is bounded below on $\mathcal{F}$, it follows that $\lim_{\jmath\rightarrow\infty}\Psi(x^{\jmath})$ exists and
  \begin{equation}\label{psilim}
    \lim_{\jmath\rightarrow\infty}[\Psi(x^{\jmath+1})-\Psi(x^{\jmath})]=\lim_{\jmath\rightarrow\infty}\|x^{\jmath+1}-x^{\jmath}\|=0.
  \end{equation}
  \par Since the sequence $\{x^{\jmath}\}$ is contained in the bounded level set $L(x^0)$ and $\mathcal{F}$ is closed and convex, the sequence has at least one accumulation point. Let $\{x^{\jmath}\}_{\jmath\in\mathcal{K}}$ be a subsequence converging to a limit $\bar{x}$, then $\bar{x}\in\mathcal{F}$ by the closedness of $\mathcal{F}$. Since the elements in $\mathcal{D}(\bar{x}_i)$ for $i=1,\ldots,n$ are finite, by the update rule of the algorithm, there exists $d=(d_1,\cdots,d_n)^T$ with $d_i\in\mathcal{D}(\bar{x}_i)$, $i=1,\ldots,n$, such that
  \begin{equation*}
    \begin{split}
       &\Psi(x^{\jmath+1})+\frac{1}{2}\|x^{\jmath+1}-x^{\jmath}\|^2\\
                \leq& f_1(x)-\sum_{i=1}^{n}(\theta_{d_i}(x^{\jmath}_i)+\theta'_{d_i}(x^{\jmath}_i)(x_i-x^{\jmath}_i))+\frac{1}{2}\|x-x^{\jmath}\|^2, \forall x\in\mathcal{F}.
    \end{split}
  \end{equation*}
  Taking the limits of both sides of above inequalities, it yields that there exists $d=(d_1,\cdots,d_n)^T$ with $d_i\in\mathcal{D}(\bar{x}_i)$, $i=1,\ldots,n$, such that
  \begin{equation*}
    \Psi(\bar{x})\leq f_1(x)-\sum_{i=1}^{n}(\theta_{d_i}(\bar{x}_i)+\theta'_{d_i}(\bar{x}_i)(x_i-\bar{x}_i))+\frac{1}{2}\|x-\bar{x}\|^2, \forall x\in\mathcal{F}.
  \end{equation*}
  The right side of above inequality is strongly convex on $\mathcal{F}$, which implies that $\bar{x}$ is a lifted stationary point of problem (\ref{PALDC}).\qed
  \end{proof}

  \begin{remark}
     In Algorithm \ref{alg:C}, if we follow Algorithm 1 in \cite{s68} where all $d_i\in\mathcal{D}(x^{\jmath}_i)$ for $i=1,\ldots,n$ have been taken into consideration, then the accumulation point of the sequence generated by Algorithm \ref{alg:C} can be proved to be a d-stationary point of problem (\ref{PALDC}). Furthermore, Algorithm \ref{alg:B} will converge to a d-stationary point of problem (\ref{Rvsig}) and Algorithm \ref{alg:A} will converge to an MPCC d-stationary point of problem (\ref{Rv}) with the same assumptions, respectively. However, for the convenience of computation, we only consider that there exists $d=(d_1,\cdots,d_n)^T$ with $d_i\in\mathcal{D}(x^{\jmath}_i)$, $i=1,\ldots,n$, such that $x^{\jmath+1}$ is a minimizer of (\ref{DCset}) in this paper.
   \end{remark}
To end this subsection, we use Figure \ref{figure2} to describe the relations between  Algorithms \ref{alg:A}, \ref{alg:B} and \ref{alg:C}.
\begin{figure}[H]
\centering
\tiny
\begin{framed}
\begin{tikzpicture}
  \matrix (m) [matrix of math nodes,row sep=1em,column sep=1.5em,minimum width=0.5em]
  {
       \text{Problem (\ref{Rv})} &\text{} & \text{Problem (\ref{Rvsig})} & \text{} & \text{Problem (\ref{PAL})} & \text{} & \text{Problem (\ref{DCset})} \\
       \text{} & \text{Alg. \ref{alg:A}} & \text{} & \text{Alg. \ref{alg:B}} & \text{} & \text{Alg. \ref{alg:C}} & \text{}\\};
  \path[-stealth]
    (m-1-1) edge [double,->] node [above] {{\tiny Relaxation}} (m-1-3)
    (m-1-3) edge [double,->] node [above] {{\tiny ALM}} (m-1-5)
    (m-1-5) edge [double,->] node [above] {{\tiny DCA}} (m-1-7)
            edge [double,->] node [below] {{\tiny Assump. Pro. \ref{PSDset}}} (m-1-7);
    \path[dotted,-,font=\scriptsize]
    (m-1-1) edge node {} (m-2-2)
    (m-2-2) edge node {} (m-1-3)
    (m-1-3) edge node {} (m-2-4)
    (m-2-4) edge node {} (m-1-5)
    (m-1-5) edge node {} (m-2-6)
    (m-2-6) edge node {} (m-1-7);
    \end{tikzpicture}
\leftline{Alg. \ref{alg:A} solves problem (\ref{Rv}) by updating relaxation parameter $\sigma$;}
\leftline{Alg. \ref{alg:B} solves problem (\ref{Rvsig}) with a fixed $\sigma$, where problem (\ref{PAL}) is its subproblem;} \leftline{Alg. \ref{alg:C} solves DC problem (\ref{PAL}) under assumptions of Proposition \ref{PSDset}, where problem (\ref{DCset}) is its subproblem. }
\end{framed}
\caption{The relations between Algorithms \ref{alg:A}, \ref{alg:B} and \ref{alg:C}}
\label{figure2}
\end{figure}
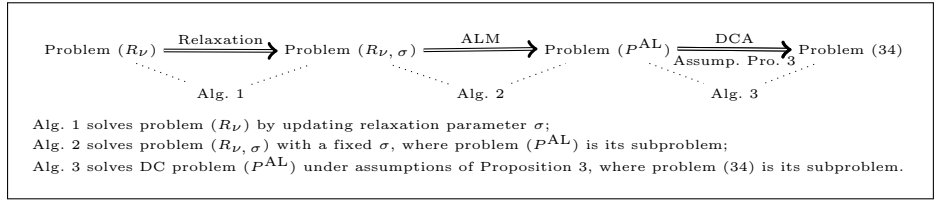

\subsection{Equilibrium problems in market modeling}
 We generalize the equilibrium problem in \cite{s58} and assume that there are $n_1$ producers and $m_1$ products. Let $c_i, a_i, b_i, f_i, s_i\in\mathbb{R}^{m_1}$, the $j$th, $j=1,2,\ldots,m_1$ elements of these vectors are the cost function, the no-production fixed cost, the linear production cost and the quantities sold by the producer in the futures and the spot markets with respect to product $j$ for producer $i$, respectively. Consider producer $i$ ($i=1,\ldots,n_1$) with the linear production cost 
\begin{equation*}
  c_{i}=a_{i}+{\rm diag}(b_{i})(s_{i}+f_{i}).
\end{equation*}
\par Consider a single price for both the spot and the futures markets, which depends linearly on the total production, i.e.,
\begin{equation*}
  \pi=\gamma-{\rm diag}(\beta)\sum_{i=1}^{n_1}(s_{i}+f_{i}),
\end{equation*}
where $\pi\in\mathbb{R}^{m_1}$ is the actual price, $\gamma\in\mathbb{R}^{m_1}_{++}$ is the no-demand price and $\beta\in\mathbb{R}^{m_1}_{++}$ is the price-demand slope of $m_1$ different products.
\par The profit maximization (minus profit minimization) problem of producer $i$ is
  \begin{equation}\label{produceri}
  \begin{split}
     \min_{s_{i},f_{i}}\ & -\pi^T\sum_{i=1}^{n_1}(s_{i}+f_{i})+e^Tc_{i}\\
       {\rm s.t.}\ & \pi=\gamma-{\rm diag}(\beta)\sum_{i=1}^{n_1}(s_{i}+f_{i}),\\
                   & {\rm diag}(w_{i}^s)s_{i}+{\rm diag}(w_{i}^f)f_{i}\leq \bar{q}_{i},
  \end{split}
\end{equation}
where $e\in\mathbb{R}^{m_1}$ is the vector whose all entries are one, $\bar{q}_{i}\in\mathbb{R}^{m_1}$ ($i=1,\ldots,n_1$) is a given bound imposed on $m_1$ products to be sold in the market (i.e., a capacity constraint) by producer $i$, and $w_{i}^s, w_{i}^f\in\mathbb{R}^{m_1}_{++}$ ($i=1,\ldots,n_1$) are given effective weight of quantities of $m_1$ products sold by producer $i$ in the spot market and future market, respectively.
\par For $n_1$ producers, we use KKT conditions to characterize the profit maximization problems for each producer. Thus, the whole equilibrium problem can be written as VLCS (\ref{VLCP}) where
\begin{equation*}
  A=(ee^T+I)\otimes B,\ C={\rm diag}(D_i),\ b=-\begin{pmatrix}
                                                            \gamma-b_1 \\
                                                            \gamma-b_2 \\
                                                             \vdots \\
                                                            \gamma-b_{n_1}
                                                           \end{pmatrix}, d=-\begin{pmatrix}
                                                             \bar{q}_1 \\
                                                             \bar{q}_2 \\
                                                             \vdots \\
                                                             \bar{q}_{n_1}
                                                           \end{pmatrix}
\end{equation*}
with
\begin{equation*}
 B_{lk}=\left\{
\begin{aligned}
 &  -\beta_l, && \ k=2l {\rm\ or\ }k=2l-1,\\
 &  0, && {\rm otherwise},
\end{aligned}
\right.
,   (D_i)_{lk}=\left\{
\begin{aligned}
 &  -w_{il}^s, && \ k=2l,\\
 &  -w_{il}^f, && \ k=2l-1,\\
 &  0, && {\rm otherwise},
\end{aligned}
\right.
\end{equation*}
for $k=1,\ldots,2m_1, l=1,\ldots,m_1, i=1,\ldots,n_1$. Here, “$\otimes$” denotes the Kronecker product of two matrices.

\subsection{Numerical experiments}

\par The AL method and DC algorithm are both coded in MATLAB and all computations are performed on a laptop (2.90 GHz, 32.0 GB RAM) with MATLAB R2020b. In the following tests, we set  $x^0\in\argmin_{x\in\mathcal{F}}(Ax-b)^\top(Cx-d)+\frac{1}{2}\|x\|_1$ and $\sigma_k=10^{-k}*\min\{(Ax^0-b)^\top (Cx^0-d),10^3\}$.
We terminate the AL algorithm with DC algorithm when the following conditions hold
$$
\|x^k-x^{k-1}\|\leq 10^{-8}, \,\, \|\min(Ax^k-b,Cx^k-d)\|_{\infty}\leq 10^{-3},
$$
$$\|(|x^k|-10^{-6}e)_+\|_0=\|(|x^{k-1}|-10^{-6}e)_+\|_0.$$

In our numerical test, instead of using a fixed $\nu$, we use an updating scheme
\begin{equation}\label{updating}
\tilde{\nu}_k=\max (1- \frac{k+1}{K}, \nu),  \,\,\,\, k=0,1,...,
\end{equation}
for $K=5, 15, 20, 30.$
Since $\tilde{\nu}_k=\nu$ after certain iterations,
the convergence results of Algorithms 1-3 still hold.  Specially, if $K=1$, then  $\tilde{\nu}_k =\nu$ for $k=0,1,\ldots $.
\subsubsection{Sparse solution of VLCS arising from optimization problem (\ref{produceri}) }
\par In our first experiment, let $n_1=5$, $m_1=10$, $\beta_l=\texttt{randi}([5,10])/10$ and $w_{il}^s=w_{il}^f=1$ for $l=1,\ldots,m_1,\ i=1,\ldots,n_1$ to generate $A$ and $C$. Since $ee^T+I $ and
$D^TB$ are positive semi-definite,  by \cite[Corollary 4.2.13]{Horn},
$A^TC=(ee^T+I)\otimes (D^TB)$ is positive semi-definite.
\par Let $x^{true}$ be a given vector.  Let $\mathcal{I}_G(x^{true})$ and $\mathcal{I}_H(x^{true})$ be two sets containing $s$ unique integers selected randomly from $\{1,2,\ldots,m\}$, respectively, and $\{1,2,\ldots,m\}\subseteq\mathcal{I}_G(x^{true})\cup\mathcal{I}_H(x^{true})$. Let $b_i=(Ax^{true})_i$ if $i\in \mathcal{I}_G(x^{true})$ and $b_i= (2Ax^{true})_i$ if  $i \not\in \mathcal{I}_G(x^{true})$. Similarly, let $d$ be a vector whose elements are equal to that of  $Cx^{true}$ for indices in $\mathcal{I}_H(x^{true})$ and equal to that of $2Cx^{true}$ for indices not in $\mathcal{I}_H(x^{true})$. We compare our algorithm with the $L_p$-minimization method proposed in \cite{s44}.
We use the active-set method implemented in KNITRO \cite{s66} with default settings to solve the following problem:
\begin{equation*}
  \begin{split}
     \min\ & \|x\|_p^p \\
       {\rm s.t.}\ & Ax-b\geq0,\ Cx-d\geq0,\ (Ax-b)^T(Cx-d)\leq\sigma_k.
  \end{split}
\end{equation*}
 \par Denote $x^*$ the solution obtained by our algorithm or the $L_p$-minimization method. “nnz” is the number of non-zero entries of $x^{true}$. “Res”$:=$ $\|\min(Ax^*-b,Cx^*-d)\|_{\infty}$. Denote $\mathcal{I}:=\{i: x^{true}_i=x^*_i=0, i=1,2,\ldots,n\}$. Then, let $s=30$ such that $|\mathcal{I}_G(x^{true})|=|\mathcal{I}_H(x^{true})|=30$ and $|\mathcal{I}_{00}(x^{true})|=10$. In the tests below, we consider the case with 5 producers and 10 products ($m=50,n=100$), and generate 10 random instances for each nnz. The computational results reported in Table \ref{part1:table1} are averaged over the 10 instances.
\begin{table}[H]
\footnotesize
\centering
\caption{Comparison of our algorithm and the $L_p$-minimization method with 5 producers, 10 products ($m=50$, $n=100$), $|\mathcal{I}_{00}(x^{true})|=10$ and $|\mathcal{I}_G(x^{true})|=|\mathcal{I}_H(x^{true})|=30$}
\begin{tabular}{c||c||c|c|c|c||c|c|c|c}
\hline
\multirow{2}{*}{nnz} & \multirow{2}{*}{} & \multicolumn{4}{c||}{$\nu(K=20)$}                             & \multicolumn{4}{c}{$p$}                         \\ \cline{3-10}
                                  &                   & 0.0004    & 0.004     & 0.04      & 0.1                & 0.1       & 0.3       & 0.5       & 1.0        \\ \hline
\multirow{3}{*}{30}               & $n-\|x^*\|_0$       & 50.1        & \textbf{50.3}        & 49.7          & 49.9   & 31.7        & 31.8        & 32.1       & 49.6         \\ 
                                  & $|\mathcal{I}|$             & 50.1       & \textbf{50.3}        & 49.7          &  49.9       & 31.1        & 31.6        & 29.4        & 49.6         \\ 
                                  & Res               & 3.8e-4 & 3.4e-4    & 2.8e-4  & 3.2e-4                   & 4.1e-4 & 5.8e-4 & 4.9e-4 & 6.7e-4  \\ \hline
\multirow{3}{*}{50}               & $n-\|x^*\|_0$       & \textbf{27.9}        &27.6      & 27.5        & 27.6                 & 16.8        & 16.2        & 26.5        & 21.6         \\ 
                                  & $|\mathcal{I}|$             & \textbf{27.9}       & 27.6        & 27.5       &27.6                &15.8        & 15.2        & 20        & 21.6         \\ 
                                  & Res               & 4.2e-4 & 3.0e-4 & 2.9e-4 & 2.5e-4          & 4.5e-4 & 4.5e-4 & 6.4e-4 & 6.8e-4 \\ \hline
\multirow{3}{*}{70}               & $n-\|x^*\|_0$       & 11.2       & 10.4        & 10.5        & 10.3 & 7.4        & 10.2        & \textbf{26.2}        & 8.8         \\ 
                                  & $|\mathcal{I}|$             & 11.2        & 10.4        & 10.5        & 10.3                   & 6.5        &8.4        & \textbf{12.8}        & 8.8         \\ 
                                  & Res           & 3.6e-4    & 3.3e-4    & 3.0e-4    &3.9e-4             & 4.1e-4 & 5.9e-4     & 5.5e-4 & 7.5e-4 \\ \hline
\end{tabular}
\label{part1:table1}
\end{table}

\subsubsection{Randomly generated problem (\ref{NEP0})}
In our second experiment, we consider randomly generated problem (\ref{NEP0}). Let components of matrix $A$ be generated by the uniform distribution in $(-20,20)$, $k$ be the largest integer smaller than or equal to $m/2$ and $B=\texttt{rand}(k,m)$. Let $s$ be the largest integer smaller than or equal to $k/3$ and $ind$ be a set containing $s$ unique couples selected randomly from $\{(1,1),(2,2),\ldots,(k,k)\}$. Let $M\in\mathbb{R}^{k\times k}$ be a diagonal matrix whose components in $ind$ are generated by the uniform distribution in $(0,1)$ and other components are equal to 0. Let $C=B^TMBA$.
\par Let $x^{true}$ be a given vector. Then, we use the same approach in the last subsection to generate vectors $b$ and $d$. Besides, when VLCS reduces to LCP, where $m=n$ and $A$ is an identity matrix, our random generated approach is consistent with that in \cite{s3,s44} for generating LCP.
\par In the tests below, let $|\mathcal{I}_G(x^{true})|=|\mathcal{I}_H(x^{true})|=30$ and $|\mathcal{I}_{00}(x^{true})|=10$. We consider the cases when $(m,n)=(50,100)$ and $(m,n)=(100,200)$, and generate 10 random instances for each case and each nnz. The computational results reported in Tables \ref{part1:table2} and \ref{part1:table3} are averaged over the 10 instances.

\begin{table}[H]
\footnotesize
\centering
\caption{Comparison of our algorithm and the $L_p$-minimization method with $m=50$, $n=100$, $|\mathcal{I}_{00}(x^{true})|=10$ and $|\mathcal{I}_G(x^{true})|=|\mathcal{I}_H(x^{true})|=30$}
\begin{tabular}{c||c||c|c|c|c||c|c|c|c}
\hline
\multirow{2}{*}{nnz} & \multirow{2}{*}{} & \multicolumn{4}{c||}{$\nu(K=20)$}                             & \multicolumn{4}{c}{$p$}                         \\ \cline{3-10}
                                  &                   & 0.0004    & 0.004     & 0.04      & 0.1                & 0.1       & 0.3       & 0.5       & 1.0        \\ \hline
\multirow{3}{*}{30}               & $n-\|x^*\|_0$       & \textbf{66}       & 62.7        & 61.5       & 64 & 22.1        & 19.6        & 27.7        & 40.4         \\ 
                                  & $|\mathcal{I}|$             & \textbf{58.5}       & 53.8        & 52.8        &     53.6         & 17.8       & 16       & 22.7        & 33.5        \\ 
                                  & Res               &  2.1e-4  & 5.2e-4    &4.2e-4    &4.2e-4                 & 2.2e-5 & 5.4e-6 & 8.5e-5 & 6.7e-7  \\ \hline
\multirow{3}{*}{50}               & $n-\|x^*\|_0$       & 63.2       & \textbf{63.8}       & 63       & 63.8                & 17.5        & 14.5        & 19.3        & 32.1        \\ 
                                  & $|\mathcal{I}|$             & 37.6       &\textbf{38.3}       & 37.6        &37.6                & 11        & 8.5       & 13        & 20         \\ 
                                  & Res               & 5.4e-4 & 3.7e-4 & 4.1e-4 & 3.5e-4          & 8.6e-7 & 1.4e-5 & 1.4e-4 & 2.2e-4 \\ \hline
\multirow{3}{*}{70}               & $n-\|x^*\|_0$       & 53.3       & \textbf{59.3}       & 53.7        & 56.8 & 11.6        & 12.9       & 23.6        & 33.4         \\ 
                                  & $|\mathcal{I}|$             & 20.6        &  \textbf{22.3}        &  20.8         &       21.6               &5        & 5.9       & 10.1        & 14.4         \\ 
                                  & Res               & 1.2e-4    & 5.1e-4    &2.6e-4    &9.1e-5    & 6.6e-6 & 1.6e-6     & 1.2e-10 & 1.9e-10 \\ \hline
\end{tabular}
\label{part1:table2}
\end{table}

\begin{table}[H]
\footnotesize
\centering
\caption{Comparison of our algorithm and the $L_p$-minimization method with $m=100$, $n=200$, $|\mathcal{I}_{00}(x^{true})|=20$ and $|\mathcal{I}_G(x^{true})|=|\mathcal{I}_H(x^{true})|=60$}
\begin{tabular}{c||c||c|c|c|c||c|c|c|c}
\hline
\multirow{2}{*}{nnz} & \multirow{2}{*}{} & \multicolumn{4}{c||}{$\nu (K=20)$}                             & \multicolumn{4}{c}{$p$}                         \\ \cline{3-10}
                                  &                   & 0.0004    & 0.004     & 0.04      & 0.1                & 0.1       & 0.3       & 0.5       & 1.0        \\ \hline
\multirow{3}{*}{60}               & $n-\|x^*\|_0$       &  135.8      & \textbf{137.6}       & 136.3        &135.8 & 33.5        & 38.1        & 56.3        & 62.3         \\ 
                                  & $|\mathcal{I}|$             & 112        &  \textbf{113}       & 112.5        &   112.0              & 27.5       & 31        & 45.7        & 50.3         \\ 
                                  & Res               & 8.6e-4 & 5.4e-4    & 2.8e-4    &5.3e-4                 & 2.6e-7 & 1.0e-5 & 2.5e-10 & 4.6e-6  \\ \hline
\multirow{3}{*}{100}               & $n-\|x^*\|_0$       & 105.2      & 133.2       & \textbf{133.2}        & 133.2               &17.5        &34.2       & 40.3        & 60.8         \\ 
                                  & $|\mathcal{I}|$    & 63.2       & 79.6        & \textbf{79.6}        & 79.6                & 10.4       & 20.9        & 24.3        & 37.3         \\ 
                                  & Res               & 5.1e-4 & 8.2e-4 & 1.8e-4 & 6.3e-4         & 8.8e-4 & 6.5e-4 & 1.1e-3 & 9.8e-4 \\ \hline
\multirow{3}{*}{140}               & $n-\|x^*\|_0$       & 90.0        & \textbf{123.7}        & 112.2       & 106.8 & 30.3        & 39.7        & 58.6        &65.9         \\ 
                                  & $|\mathcal{I}|$             & 33.3        & 45.0       & 42.2       &      \textbf{48.2}              & 11.7        & 15.4        &21.3        & 24.5        \\ 
                                  & Res               & 4.3e-5    & 7.2e-5    & 2.6e-3    &1.0e-3    & 6.9e-9 & 5.8e-10     & 1.0e-9 & 9.1e-5 \\ \hline
\end{tabular}
\label{part1:table3}
\end{table}

\begin{figure}[H]
  \includegraphics[width=\linewidth]{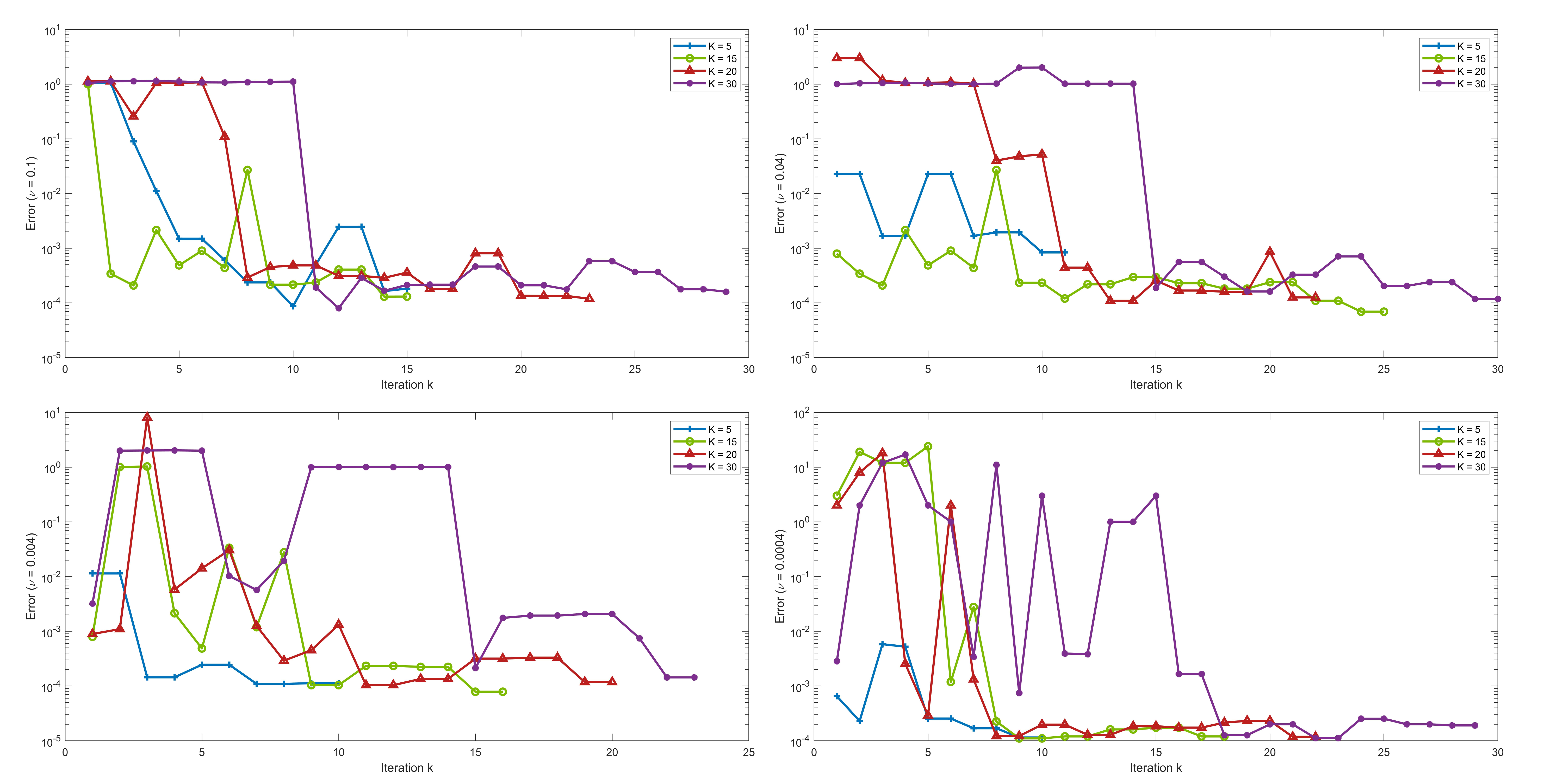}
  \caption{Numerical results with different $\nu$ and $K$ where $m=100$, $n=200$, nnz=100, $|\mathcal{I}_{00}(x^{true})|=20$ and $|\mathcal{I}_G(x^{true})|=|\mathcal{I}_H(x^{true})|=60$}
  \label{fig:2}
\end{figure}

Figure \ref{fig:2} shows the numerical results of the proposed algorithm with the updating scheme (\ref{updating}), where Error:=Res$+[\Phi(x^*)-\Phi(x^{true})]_+$.

\par From all numerical results above in Tables 1, 2 and 3,  it is obvious that our algorithm can always obtain more sparse solutions than the $L_p$-minimization method. On the other hand, in order to show the reliability of our algorithm, we consider different values of $|\mathcal{I}_G(x^{true})|$, $|\mathcal{I}_H(x^{true})|$, and $|\mathcal{I}_{00}(x^{true})|$ in randomly generated problems. See Appendix B. From our numerical experiments, we find that the numerical performance of our algorithm with the updating scheme (\ref{updating}) is much better than using a fixed parameter $\nu$.

\section{Conclusion}
\par In this paper, we define a continuous relaxation problem (\ref{Rv}) that has the same global minimizers and optimal value with problem (\ref{P0}) when $\nu$ is a sufficiently small positive number. We introduce MPCC lifted-stationarity of problem (\ref{Rv}) and establish the relationship among several necessary optimality conditions of problem (\ref{Rv}) in Theorem \ref{NecOptCon} based on different constraint qualifications. In Theorem \ref{Errorbound}, we provide an upper bound of the distance between a feasible point of problem (\ref{Rvsig}) and the feasible set of problem (\ref{P0}). Moreover, we propose an approximation method (Algorithm \ref{alg:A}) to solve problem (\ref{Rv}) and an AL method (Algorithm \ref{alg:B}) to solve its subproblem (\ref{Rvsig}). In Theorem \ref{AL}, we prove that Algorithm \ref{alg:B} converges to a lifted stationary point of subproblem (\ref{Rvsig}) under RCPLD condition and regularity. In Theorem \ref{algAMPCClifted}, we prove that Algorithm \ref{alg:A} converges to an MPCC lifted-stationary point of problem (\ref{Rv}) under MPCC-LICQ. To compare our algorithm with the $L_p$-minimization method for finding sparse solutions of complementarity problems, numerical results for problem (\ref{P0}) with VLCS constraints are presented, which show that our algorithm can find more sparse solutions than the $L_p$-minimization method.\\

\noindent\textbf{Appendix A}.
{\em Proof of Lemma \ref{EP0Rv}}:
 By Assumption \ref{assump1}, let $k>0$ be the global minimum of (\ref{P0}). For an integer $s$ with $0\leq s\leq n$, denote $Q_s:=\{x\in\mathbb{R}^n:\|x\|_{0}\leq s\}$ and ${\rm dist}(\mathcal{S},Q_s):=\inf\{{\rm dist}(x,Q_s):x\in \mathcal{S}\}$. Then the set $\mathcal{S}$ does not have a vector $x$ such that $\|x\|_0<k$, which means ${\rm dist}(\mathcal{S},Q_{k-K})>0$ for all $K=1,\ldots,k$. Define
\begin{equation}\label{2.1}
  \bar{\nu}=\min\left\{\frac{1}{K}{\rm dist}(\mathcal{S},Q_{k-K}):K=1,\ldots,k\right\}.
\end{equation}
  \par (i) Firstly, we prove that a global minimizer $x^*$ of problem (\ref{P0}) is also a global minimizer of (\ref{Rv}) for any $0<\nu<\bar{\nu}$. Since the global optimality of (\ref{P0}) yields $\|x\|_0\geq k$ for any $x\in \mathcal{S}$, we prove the conclusion by two cases.
  \par Case 1. $\|x\|_0=k$. Then, for any $i\in \Gamma(x)$,
  \begin{equation*}
    |x_i|\geq\min\{|x_j|>0:j=1,\ldots,n\}={\rm dist}(x,Q_{k-1})\geq {\rm dist}(\mathcal{S},Q_{k-1})\geq\bar{\nu},
  \end{equation*}
  where the last inequality comes from (\ref{2.1}). Due to $0<\nu<\bar{\nu}$, we obtain $|x_i|>\nu$ for all $i\in\Gamma(x)$, which means that $\Phi(x)=k=\Phi(x^*)$.
  \par Case 2. $\|x\|_0=r>k$. If $|\Gamma_2(x)|=r'\geq k$, from $\phi(t)>0$ for $t>0$ and $\phi(t)=0$ for $t=0$, we have
  \begin{equation*}
    \Phi(x)=\sum_{i\in\Gamma_2(x)}\phi(x_i)+\sum_{i\in\Gamma_1(x)}\phi(x_i)\geq k+\sum_{i\in\Gamma_1(x)}\phi(x_i)>k.
  \end{equation*}
  Now assume $r'<k$, and without loss of generality, assume $|x_1|,\cdots,|x_{r-r'}|\in(0,\nu)$. Since $r'<k$, we know from (\ref{2.1}) that $\frac{1}{k-r'}{\rm dist}(\mathcal{S},Q_{r'})\geq\bar{\nu}$. Together with
  \begin{equation*}
    |x_1|+\cdots+|x_{r-r'}|\geq\sqrt{x_1^2+\cdots+x_{r-r'}^2}\geq {\rm dist}(x,Q_{r'})\geq {\rm dist}(\mathcal{S},Q_{r'}),
  \end{equation*}
  we obtain
  \begin{equation*}
       \Phi(x)=\frac{1}{\nu}(|x_1|+\cdots+|x_{r-r'}|)+r'\geq\frac{1}{\nu}{\rm dist}(\mathcal{S},Q_{r'})+r'\geq\frac{1}{\nu}(k-r')\bar{\nu}+r'>k,
  \end{equation*}
  where the last inequality is obtained by $0<\nu<\bar{\nu}$. The above two cases imply that $\Phi(x)\geq k=\Phi(x^*)$ for all $x\in \mathcal{S}$. Thus, $x^*$ is also a global minimizer of (\ref{Rv}). Moreover, we know $\|x^*\|_0=\Phi(x^*)$ for each minimizer $x^*$ of (\ref{Rv}).
  \par (ii) Next we prove that a global minimizer $x^*$ of problem (\ref{Rv}) with $0<\nu<\bar{\nu}$ is a global minimizer of problem (\ref{P0}). Assume on the contrary $x^*$ is not a solution of (\ref{P0}). Let $\tilde{x}$ be a global minimizer of (\ref{P0}), which means $\|\tilde{x}\|_0=k$. Since $\phi(t)\leq|t|^0$, we have $\Phi(\tilde{x})\leq\|\tilde{x}\|_0$. By similar ways in the proof for Case 2 in (i), we could obtain $\Phi(x^*)>k=\|\tilde{x}\|_0\geq\Phi(\tilde{x})$ for any $0<\nu<\bar{\nu}$. This contradicts the global optimality of $x^*$ for (\ref{Rv}). Thus, $x^*$ is a global minimizer of (\ref{P0}).
  \par Therefore, when $0<\nu<\bar{\nu}$, (\ref{P0}) and (\ref{Rv}) have the same global minimizers and optimal values.\qed

~\\
\textbf{Appendix B}. Numerical results for random generated problem with different $|\mathcal{I}_G(x^{true})|$, $|\mathcal{I}_H(x^{true})|$ and $|\mathcal{I}_{00}(x^{true})|$:
\begin{table}[H]
\footnotesize
\centering
\caption{Comparison of our algorithm and the $L_p$-minimization method with $m=50$, $n=100$, $|\mathcal{I}_{00}(x^{true})|=30$ and $|\mathcal{I}_G(x^{true})|=|\mathcal{I}_H(x^{true})|=40$}
\begin{tabular}{c||c||c|c|c|c||c|c|c|c}
\hline
\multirow{2}{*}{nnz} & \multirow{2}{*}{} & \multicolumn{4}{c||}{$\nu(K=20)$}                             & \multicolumn{4}{c}{$p$}                         \\ \cline{3-10}
                                  &                   & 0.0004    & 0.004     & 0.04      & 0.1                & 0.1       & 0.3       & 0.5       & 1.0        \\ \hline
\multirow{3}{*}{30}               & $n-\|x^*\|_0$       & 64.0        & 63.8       & 62.0       & \textbf{65.0} & 27.1        & 27.6        & 32        & 37.8         \\ 
                                  & $|\mathcal{I}|$             & 54.4        & 55.0        & 51.8        &      \textbf{55.4}         & 22.8       & 23.6       & 27.2       & 31.9        \\ 
                                  & Res               & 3.9e-4 & 5.4e-4    &5.5e-4    &1.6e-4                  & 8.2e-5 & 1.9e-4 & 2.3e-4 & 5.6e-6  \\ \hline
\multirow{3}{*}{50}               & $n-\|x^*\|_0$       & 62.8        & \textbf{63.0}       &62.8       & 61.7                & 25.3        & 26.4        & 28.1       & 31.6        \\ 
                                  & $|\mathcal{I}|$             & 36.6        & 36.5        & 36.6        & \textbf{36.7}                & 17.4        & 17.9      & 19.1        & 21         \\ 
                                  & Res               & 8.1e-5 & 4.6e-4 & 2.1e-4 & 8.9e-5          & 9.3e-5 & 1.9e-5 & 8.9e-5 & 1.2e-4 \\ \hline
\multirow{3}{*}{70}               & $n-\|x^*\|_0$       & 50.4        & 60.5        & \textbf{63.5}        & 62.0 & 27.2        & 35.5       & 35.3        & 39.9         \\ 
                                  & $|\mathcal{I}|$             & 21.2        & 24.9       & \textbf{26.0}        &      25.1              &11.3        & 13.6       & 14.3        & 16         \\ 
                                  & Res               & 1.5e-4    & 1.3e-4    & 8.2e-4    &2.4e-4    & 1.1e-4 & 1.0e-4     & 1.9e-5 & 3.3e-5 \\ \hline
\end{tabular}
\label{part1:table4}
\end{table}

\begin{table}[H]
\footnotesize
\centering
\caption{Comparison of our algorithm and the $L_p$-minimization method with $m=100$, $n=200$, $|\mathcal{I}_{00}(x^{true})|=60$ and $|\mathcal{I}_G(x^{true})|=|\mathcal{I}_H(x^{true})|=80$}
\begin{tabular}{c||c||c|c|c|c||c|c|c|c}
\hline
\multirow{2}{*}{nnz} & \multirow{2}{*}{} & \multicolumn{4}{c||}{$\nu(K=20)$}                             & \multicolumn{4}{c}{$p$}                         \\ \cline{3-10}
                                  &                   & 0.0004    & 0.004     & 0.04      & 0.1                & 0.1       & 0.3       & 0.5       & 1.0        \\ \hline
\multirow{3}{*}{60}               & $n-\|x^*\|_0$       & 80.4     & \textbf{125.8}       &122.8       & 125.8 & 44.7        & 62        & 58.7        & 71.8         \\ 
                                  & $|\mathcal{I}|$             & 69.2      &  \textbf{107.4}      & 106.5        &      107.4              & 37.6       & 52.6        & 49.2        & 60.5         \\ 
                                  & Res               & 2.1e-3 & 8.9e-4    & 9.5e-4    &2.2e-3                 & 1.3e-5 & 8.0e-6 & 1.6e-4 & 9.5e-5  \\ \hline
\multirow{3}{*}{100}               & $n-\|x^*\|_0$       &83.5     & 114.3        & \textbf{120.2}       & 111.3                &35        &30.8      & 49.7        & 52.7         \\ 
                                  & $|\mathcal{I}|$    & 50.3       & 69.0       & \textbf{69.2}        & 67.2                & 23.0       & 20.8        & 30.7       & 32         \\ 
                                  & Res               & 3.5e-4 & 5.1e-4 & 9.5e-4 & 7.0e-4        & 1.7e-4 & 1.4e-7 & 4.4e-6 & 6.5e-9 \\ \hline
\multirow{3}{*}{140}               & $n-\|x^*\|_0$       & 100.8       & 116.7        & \textbf{119.8}       & 114.4 & 27.5        & 24.7        & 67.0       &51.5        \\ 
                                  & $|\mathcal{I}|$             & 40.2        & 46.2       & \textbf{48.6}        &      46.6              & 9.3        & 9.2        &\textbf{26.8}        & 20.2       \\ 
                                  & Res               &7.2e-4    & 7.8e-4    & 8.1e-4    &4.5e-4   & 1.3e-8 & 2.0e-5     & 1.4e-8 & 2.0e-7 \\ \hline
\end{tabular}
\label{part1:table5}
\end{table}

\begin{acknowledgements}
The authors would like to thank Prof. William Hager, the coordinating editor  and two referees for their helpful comments.
\end{acknowledgements}

\noindent
{\bf Funding } This work is supported in part by Hong Kong Research Grant Council PolyU15300120.\\
{\bf Data availability} The data that support the findings of this
study are available from the corresponding author upon request.\\
{\bf Conflict of interest}

 The authors declare that they have no conflict of interest.

\end{document}